\documentclass[12pt]{amsart}
\usepackage{amssymb,latexsym,mathtools}
\usepackage[margin=1.25in]{geometry}
\usepackage{mathrsfs}
\usepackage{mathtools}
\usepackage{graphicx}
\usepackage{esint}
\usepackage{braket}
\usepackage{url, hyperref}
\newtheorem{theorem}{Theorem}
\newtheorem{lemma}{Lemma}
\newtheorem{definition}{Definition}

\newtheorem{prop}{Proposition}

\newtheorem{remark}{Remark}
\newtheorem{corollary}{Corollary}

\newcommand{\vs}{\vskip.075in}
\newcommand{\be}{\begin{equation}}
\newcommand{\ee}{\end{equation}}
\newcommand{\R}{\mathbb{R}}

\begin{document}

\title[Discontinuous Operators Compatible with Finsler Norms]{Comparison Principles for Second-Order Elliptic/Parabolic  Equations with Discontinuities  in the Gradient Compatible with Finsler Norms}
\author[P.S.\ Morfe and P.E.\ Souganidis]{Peter S.\ Morfe and Panagiotis E. Souganidis}

\begin{abstract}  This paper is about elliptic and parabolic partial differential operators with  discontinuities in the gradient which  are compatible with a Finsler norm in a sense to be made precise.  Examples  of this type of problems arise in a number of contexts, most notably the recent work of Chatterjee and the second author \cite{chatterjee_souganidis} on scaling limits of discrete surface growth models as well as $L^{\infty}-$variational problems.  Building on the approach of Ishii \cite{ishii}, new comparison results are proven within a unified framework that includes a number of previous results as special cases. \end{abstract}

\date{\today}
\maketitle

\section{Introduction}  This paper is about the comparison principle for possibly degenerate  second-order elliptic and parabolic operators with discontinuities in the gradient that are compatible with a Finsler norm in a sense that will be defined below.  Such equations arise in a number of applications involving anisotropic geometries, including scaling limits of deterministic growth models on lattices studied recently by  Chatterjee and the second author \cite{chatterjee_souganidis} as well as $L^{\infty}-$variational problems--see, for example,  Ishibashi and Koike~\cite{ishibashi_koike}, Belloni, Juutinen  and Kawohl \cite{belloni_kawohl_juutinen}, Belloni and Kawohl~\cite{belloni_kawohl}, Di Castro, P\'{e}rez-Llanos and Urbano~\cite{di-castro_perez-llanos_urbano}, P\'{e}rez-Llanos and Rossi~\cite{ perez-llanos_rossi}, and Rossi and Saez~\cite{rossi_saez}. 
\vs
It is a fact well understood in the theory of viscosity solutions that, in order to deal with equations with discontinuities in the gradient, it is enough to come up with an appropriate class of tests functions.  The goal is to find new test functions that ``resolve'' the discontinuities, that is, along their  gradients and Hessians, the nonlinearities become continuous. This is heuristically what we mean by compatibility. 
\vs
The contribution of our paper is that we are able to produce such test functions for a large class of new equations while also covering some older results.  Later in this introduction we discuss in detail the literature and the connections with our work.

\vs
To put  the paper into context, we present next two of the most basic examples, which, in spite of their simplicity, still require the full power of our results. The first is about surface growth models and the second about Finsler infinity Laplacians.

\subsection*{Surface growth models}  
We consider  an elementary discrete parabolic equation in $\mathbb{Z}^{d}$ similar to the Kolmogorov equation for the simple random walk.  Given an initial function $u^{(0)} : \mathbb{Z}^{d} \to \mathbb{R}$ and an $\alpha \in [1,\infty)$, we define the evolution $(u^{(n)})_{n \in \mathbb{N} \cup \{0\}}$ by 
\be\label{takis1}
u^{(n+1)}(x) - u^{(n)}(x) =  M_{\alpha}(\{u^{n}(y) - u^{n}(x) \, :  \, \| y - x \| = 1\}),
\ee
where, for a finite  set $A=\{a_1,\ldots,a_d\}$, 
\be\label{E: median flow}
{M}_{\alpha}(A) = \left\{ \begin{array}{r l}
						\text{Med}(A)=\text{the median of $A$}, & \text{if} \, \, \alpha = 1, \\
						{\text{argmin}}_{y\in \R} \sum_{i=1}^d |a_i-y|^{\alpha}, & \text{otherwise.}
					\end{array} \right.
\ee

Recently, it was proved in \cite{chatterjee_souganidis}, among other things, that the parabolic scaling limit of \eqref{takis1}
is described by  the PDE
	\begin{equation} \label{E: l1 median}
		u_{t} - F_{\alpha}(Du,D^{2}u) = 0 \ \ \text{in} \ \ \mathbb{R}^{d} \times (0,\infty),
	\end{equation}
where the operator $F_{\alpha}$ is (partially) defined by
	\begin{align} 
		F_{1}(p,X) &= X_{ii} \quad \text{if} \, \, |p_{i}| < \min \left\{ |p_{1}|, \dots, |p_{i-1}|,|p_{i+1}|,\dots,|p_{d}| \right\}, \label{E: median operator} \\
		F_{\alpha}(p,X) &= \left(\sum_{j = 1}^{d} |p_{j}|^{\alpha - 2} \right)^{-1} \sum_{i = 1}^{d} |p_{i}|^{\alpha - 2} X_{ii} \quad \text{if} \, \, \alpha \in (1,\infty). \label{E: p laplace}
	\end{align}
Evidently, any extension of $F_{1}$ in  \eqref{E: median operator} to $\mathbb{R}^{d} \times \mathcal{S}^{d}$, where $\mathcal{S}^{d}$ is the space of $d\times d$ symmetric matrices,   is necessarily discontinuous. Nonetheless, the set of discontinuities has a particular structure.  It turns out that $F_{1}$ is compatible with a certain norm that is inherited from the piecewise linear geometry of $\mathbb{Z}^{d}$.  As we will show below, this fact allows us to prove that the Cauchy problem for \eqref{E: l1 median} is well-posed, which is one of the requirements for the scheme to converge to its solution.
\vs
The nonlinearities \eqref{E: p laplace} are reminiscent of the $p$-Laplace operators. 

When $\alpha \geq 2$, uniqueness follows by a standard argument since the operator is continuous away from $\{0\} \times \mathcal{S}^{d}$.  In the regime $\alpha < 2$, the discontinuities are more severe and the comparison question requires new ideas.  Nonetheless, it turns out that these nonlinearities are well-adapted to the geometry of a certain norm --- in fact, the same norm as \eqref{E: median operator} --- and, therefore, our approach applies directly.  

\subsection*{Finsler infinity Laplacian}  Notice that letting  $\alpha \to -\infty$ in \eqref{E: p laplace}  recovers $F_{1}$ as in \eqref{E: median operator} once again.  On the other hand, the limit $\alpha \to \infty$ yields  instead an equation which, if we let $\varphi$ denote the $\ell^{1}-$norm, can be written as
	\begin{equation} \label{E: l1 infinity caloric}
		u_{t} - \langle D^{2}u \cdot \partial \varphi^{*}(Du), \partial \varphi^{*}(Du) \rangle = 0 \ \ \text{in} \ \ \mathbb{R}^{d} \times (0,\infty).
	\end{equation}
Following \cite{juutinen_kawohl}, we will refer to this as the $\ell^{1}-$infinity caloric equation.
\vs	
The nonlinearity  appearing in \eqref{E: l1 infinity caloric}, which  is the (multi-valued) infinity Laplacian associated with the $\ell^{1}-$norm defined by
	\begin{equation*}
		\langle D^{2}u \cdot \partial \varphi^{*}(Du), \partial \varphi^{*}(Du) \rangle = \{ \langle D^{2}u \cdot q, q \rangle \, :  \, q \in \partial \varphi^{*}(Du)\},
	\end{equation*}
has appeared in a number of previous works, notably that of Crandall, Gunnarsson, and Wang \cite{crandall_gunnarsson_wang}, who showed how to prove a comparison principle for PDE in which it appears.
\vs 
It turns out that the Finsler infinity Laplacian is another example of an operator that is compatible with a norm.  This leads to a number of new results of interest for $L^{\infty}-$variational problems, most notably comparison for PDE involving infinity Laplacian operators associated with $C^{2}-$norms, polyhedral norms, and arbitrary Finsler norms in dimension $d = 2$.

\subsection*{Compatibility with a Finsler norm} \label{S: definition of compatibility} We now make precise what we mean by compatibility with a Finsler norm.
\vs

For the sake of concreteness, we will use parabolic equations like \eqref{E: l1 median} as our principal running example, though elliptic problems will come later.  As is more-or-less standard in the viscosity theory for equations with discontinuous coefficients (recall, for example, the level set mean curvature flow), we restate the single equation as two inequalities.  Thus, the Cauchy problem for an equation like \eqref{E: l1 median} is  cast as 
	\begin{equation} \label{E: parabolic}
		\left\{ \begin{array}{r l}
				u_{t} - \overline{F}(Du,D^{2}u) \leq 0 & \text{in} \ \ \mathbb{R}^{d} \times (0,\infty), \\
				u_{t} - \underline{F}(Du,D^{2}u) \geq 0 & \text{in} \ \ \mathbb{R}^{d} \times (0,\infty), \\
				u = u_{0} & \text{on} \quad \mathbb{R}^{d} \times \{0\}.
			\end{array} \right.
	\end{equation}
The reader can imagine that $\overline{F}$ and $\underline{F}$ are respectively the upper and lower semi-continuous envelopes of the discontinuous nonlinearity in question, though, as discussed in Remark \ref{R: ambiguous},  that will usually not be completely true. 
\vs
The  first assumptions on the pair $(\overline{F},\underline{F})$ are ellipticity and semi-continuity, that is,
	\begin{gather}
		\overline{F}(p,X + Y) \geq \overline{F}(p,X) \, \, \text{if} \, \, (p,X) \in \mathbb{R}^{d} \times \mathcal{S}^{d}, \, \, Y \in \mathcal{S}^{d}, \, \, Y \geq 0, \label{A: elliptic_1} \\ 
		\underline{F}(p,X + Y) \geq \underline{F}(p,X) \, \,\text{if} \, \, (p,X) \in \mathbb{R}^{d} \times \mathcal{S}^{d}, \, \, Y \in \mathcal{S}^{d}, \, \, Y \geq 0, \label{A: elliptic_2} \\ 
		\overline{F} \in USC(\mathbb{R}^{d} \times \mathcal{S}^{d}), \quad \underline{F} \in LSC(\mathbb{R}^{d} \times \mathcal{S}^{d}), \label{A: semicontinuous}
	\end{gather}
where 	$USC(\mathbb{R}^{d} \times \mathcal{S}^{d})$ and $LSC(\mathbb{R}^{d} \times \mathcal{S}^{d})$ are respectively the set of upper-semicontinuous and lower-semicontinous functions on $\mathbb{R}^{d} \times \mathcal{S}^{d}$. 
\vs
Next, we require that the pair $(\overline{F},\underline{ F})$ is compatible with the geometry associated with some Finsler norm $\varphi$.  Recall that $\varphi : \mathbb{R}^{d} \to [0,\infty)$ is a Finsler norm if it is positively one-homogeneous, convex, and positive definite.  In what follows, we will denote by $\varphi^{*}$ the dual norm of $\varphi$.  (See Section \ref{S: Finsler norms} for complete definitions.) 
\vs
To make the definition of compatibility precise, we need to introduce some notation.  First, given $p \in \mathbb{R}^{d}$, we define the generalized tangent space $\mathcal{T}(p,\varphi)$ by
	\begin{equation*}
		\mathcal{T}(p,\varphi) = \partial \varphi^{*}(p)^{\perp}.
	\end{equation*}

Next, following Ishii \cite{ishii}, if $V$ is any linear subspace of $\mathbb{R}^{d}$,  let $\pi_{V} : \mathbb{R}^{d} \to \mathbb{R}^{d}$ denote the orthogonal projection onto $V$ and  $\mathcal{S}_{V} \subseteq \mathcal{S}^{d}$ be the subspace consisting of symmetric matrices that only ``see" $V$, that is,
	\begin{equation*}
		\mathcal{S}_{V} = \{X \in \mathcal{S}^{d} \, :  \, \pi_{V} \circ X \circ \pi_{V} = X \}.
	\end{equation*}
Here, we are interested in the spaces $\mathcal{S}(p,\varphi)$ defined by 
	\begin{equation*}
		\mathcal{S}(p,\varphi) = \mathcal{S}_{\langle p \rangle \oplus \mathcal{T}(p,\varphi)}.
	\end{equation*}
The interested reader can find explicit computations of $\mathcal{T}(p,\varphi)$ and $\mathcal{S}(p,\varphi)$ when $\varphi$ is the $\ell^{1}-$norm in Appendix \ref{Appendix: technical lemmas}.
\vs

Now we can define the notion of compatibility.

\begin{definition}  Given a Finsler norm $\varphi : \mathbb{R}^{d} \to [0,\infty)$, a pair $(\overline{F},\underline{F})$ of operators $\overline{F},\underline{F} : \mathbb{R}^{d} \times \mathcal{S}^{d} \to \mathbb{R}$ is said to be \emph{compatible with $\varphi$} if
	\begin{equation} \label{A: consistency}
		\overline{F}(p,X) = \underline{F}(p,X) \, \, \text{for each} \, \, p \in \mathbb{R}^{d} \, \, \text{and} \, \, X \in \mathcal{S}(p,\varphi).
	\end{equation}
\end{definition}  
	
As we will see below, the discrete surface growth models investigated in \cite{chatterjee_souganidis} provide a host of examples of PDE of the form \eqref{E: parabolic} that are compatible with polyhedral Finsler norms.  These include \eqref{E: l1 median} as a particular case.
\vs
Our other main example of   nonlinearities  compatible with a Finsler norm are the Finsler infinity Laplacians (see Section \ref{S: infinity laplace}), which,  among other contexts, are of interest in the study of $L^{\infty}-$variational problems.  In Proposition \ref{P: infinity laplace}, we show that  Finsler infinity Laplacians are always compatible with their associated norms.  

\subsection*{Revisiting Ishii's method}  Our main comparison results are proved using a method that is a distillation of ideas of Ishii from \cite{ishii}.  This abstract method lays out sufficient conditions under which elliptic and parabolic PDE compatible with a Finsler norm admit comparison principles.  We will explain the method before stating our main results in the next section.
\vs 
It is shown in \cite{ishii}  that in order to prove comparison, it suffices to find a smooth Finsler norm that respects the geometry in a certain sense.  Inspired by the terminology used by Crandall, Gunnarsson, and Wang \cite{crandall_gunnarsson_wang}, we call these shielding norms.  The precise definition is given next.

\begin{definition} \label{D: good test function} Given a Finsler norm $\varphi$ in $\mathbb{R}^{d}$, we say that a Finsler norm $\psi : \mathbb{R}^{d} \to [0,\infty)$ is a \emph{shielding norm for $\varphi$} if $\psi \in C^{2}(\mathbb{R}^{d} \setminus \{0\})$ and
	\begin{equation*}
		D^{2}\psi(q) \in \mathcal{S}(D\psi(q),\varphi) \quad \text{for each} \, \, q \in \mathbb{R}^{d} \setminus \{0\}.
	\end{equation*}
\end{definition}

With these definitions in hand, it is straightforward to prove the following theorem arguing as in \cite{ishii}. We present the proof in Section \ref{S: warm up}.

	\begin{theorem} \label{T: main parabolic}  Suppose that $\varphi$ is a Finsler norm in $\mathbb{R}^{d}$ that possesses a shielding norm $\psi$, and  $(\overline{F},\underline{F})$ is a pair of nonlinearities satisfying \eqref{A: elliptic_1}, \eqref{A: elliptic_2}, and \eqref{A: semicontinuous} that are compatible with the geometry of $\varphi$. 
\vs
If $(w,v) \in USC(\mathbb{R}^{d} \times [0,T)) \times LSC(\mathbb{R}^{d} \times [0,T))$ are bounded and satisfy 
		\begin{align*}
			w_{t} - \overline{F}(Dw,D^{2}w) \leq 0 \ \  \text{in} \ \  \mathbb{R}^{d} \times (0,T), \\
			v_{t} - \underline{F}(Dv,D^{2}v) \geq 0 \ \  \text{in} \ \  \mathbb{R}^{d} \times (0,T), 
			\end{align*}
and 			
\begin{align*} \lim_{\delta \to 0^{+}} \sup \left\{ v(x,0) - w(y,0) \, :  \, \|x - y\| \leq \delta \right\} \leq 0,
		\end{align*}
	then $w \leq v$ in $\mathbb{R}^{d} \times [0,T)$.
	\end{theorem}  	
There is no need to restrict attention to parabolic problems, as experts will have no trouble deducing from the proof.  See Proposition \ref{P: main elliptic} for one such extension, which is used later to apply the results to $L^{\infty}-$variational problems.

\subsection*{Main results}  We now turn to our main results, which identify two classes of Finsler norms for which it is possible to construct shielding norms.
\vs
First, we prove that it is always possible to find shielding norms for polyhedral Finsler norms.  (See Section \ref{S: Finsler norms} for the definition of polyhedral.)

\begin{theorem} \label{T: polyhedral norms} Any polyhedral Finsler norm  $\varphi$ in $\mathbb{R}^{d}$ has a shielding norm $\psi$.
Moreover, for each $\delta > 0$, there is a shielding norm $\psi_{\delta}$ for $\varphi$ such that 
	\begin{equation*}
		\sup \left\{ |\psi_{\delta}(e) - \varphi(e)| \, :  \,e \in S^{d-1} \right\} \leq \delta.
	\end{equation*}
\end{theorem}  

Combining Ishii's method with Theorem \ref{T: polyhedral norms}, we obtain comparison results for PDE compatible with polyhedral geometries.  In particular, this applies to the scaling limits studied in \cite{chatterjee_souganidis}.
\vs 
The other nontrivial setting where we know how to find shielding norms is that of $C^{2}-$Finsler norms.  This is a novel result in itself that has its own applications.

\begin{theorem} \label{T: C2 norms}  If $\varphi \in C^{2}(\mathbb{R}^{d} \setminus \{0\})$ is a Finsler norm, then $\varphi$ is a shielding norm for itself.  \end{theorem}  

As shown later in the paper, Theorem~\ref{T: C2 norms} implies that the infinity Laplacian associated with any $C^{2}-$Finsler norm can be treated much like the Euclidean infinity Laplacian, even though its set of discontinuities may be much larger than $\{0\} \times \mathcal{S}^{d}$.  This fact, which seems not to have been observed previously in the literature, has interesting ramifications for the theory of $L^{\infty}-$variational problems.
\vs
Finally, for clarity, it is worth noting the (ultimately trivial) fact that, if $\varphi^{*} \in C^{1}(\mathbb{R}^{d} \setminus \{0\})$, then the Euclidean norm serves as a shielding norm.

\begin{prop} \label{P: strictly convex norms} If $\varphi$ is a Finsler norm in $\mathbb{R}^{d}$ and $\varphi^{*} \in C^{1}(\mathbb{R}^{d} \setminus \{0\})$, then the Euclidean norm $\|\cdot\|$ is a shielding norm for $\varphi$.  Furthermore, in this case, a pair $(\overline{F},\underline{F})$ satisfying \eqref{A: elliptic_1}, \eqref{A: elliptic_2}, and \eqref{A: semicontinuous} is compatible with $\varphi$ if and only if 
	\begin{equation} \label{E: mean curvature type}
		\overline{F} \equiv \underline{F} \quad \text{in} \, \, (\mathbb{R}^{d} \setminus \{0\}) \times \mathcal{S}^{d} \quad \text{and} \quad \overline{F}(0,0) = \underline{F}(0,0).
	\end{equation}
\end{prop}
\vs
The proposition above shows that our framework includes as a special case classical examples like the $1-$Laplacian, that is, the level set mean curvature flow.
\vs
Lastly, in the appendix, we show that it is possible to prove in dimension $d = 2$ comparison in great generality.  This follows from a construction by the first author in \cite{level_set}, which is itself an improvement of a construction by Ohnuma and Sato \cite{ohnuma_sato}; see also \ \cite[Section 5.3]{crandall_gunnarsson_wang}.  Note that we do not proceed by constructing shielding norms in this setting.  We expect that, even in dimension two, shielding norms do not always exist, but it is clear how to find suitable approximations in that dimension.  
\vs
In view of the preceding discussion,  we have a good understanding of PDE compatible with a Finsler norm $\varphi$ if 
\begin{equation} \label{A: running assumptions}
\begin{split}
&\text{either} \ \ \text{(i)}~\varphi \in C^{2}(\mathbb{R}^{d} \setminus \{0\}), \quad \text{(ii)}~\varphi^{*} \in C^{1}(\mathbb{R}^{d} \setminus \{0\}),\\
&\text{(iii)} \ \  \varphi \, \, \text{is polyhedral}, \quad \text{or} \quad \text{(iv)} \, \, d = 2.
\end{split}
\end{equation}

As is generally the case in the theory of viscosity solutions, we can use the comparison results proved here to establish the existence and uniqueness of solutions of various problems.  We make that explicit for one of our main examples, the Finsler infinity caloric equation.

\begin{theorem} \label{T: infinity caloric} Suppose that $\varphi$ is a Finsler norm in $\mathbb{R}^{d}$ and \eqref{A: running assumptions} holds.  Given $u_{0} \in BUC(\mathbb{R}^{d})$ and $T > 0$, the Cauchy problem
	\begin{equation} \label{E: infinity caloric}
		\left\{ \begin{array}{r l}
			u_{t} - \langle D^{2}u \cdot \partial \varphi^{*}(Du), \partial \varphi^{*}(Du) \rangle = 0 & \text{in} \ \ \mathbb{R}^{d} \times (0,T), \\
			u = u_{0} & \text{on} \ \ \mathbb{R}^{d} \times \{0\}
		\end{array} \right.
	\end{equation}
has a unique viscosity solution $u \in BUC(\mathbb{R}^{d} \times [0,T])$.
\end{theorem}

\subsection*{Literature review}  Throughout the paper, we utilize the Crandall-Lions theory of viscosity solutions of elliptic and parabolic PDE.  We refer to the Crandall, Ishii and Lions ``User's Guide'' \cite{user} for the fundamental results of the theory and a historical overview.
\vs
A number of works undergird the treatment of discontinuous operators in this paper.  Most notably, Ishii's approach to level set PDE with discontinuous coefficients in \cite{ishii} was the inspiration for what eventually became our definition of compatibility and shielding norms.  
\vs
Other work that preceded Ishii's development of level set PDE were contributed by Evans and Spruck \cite{EvansSpruck} and Chen, Giga and Goto \cite{CGG} for mean curvature-type equations with  bounded singularity only at $p=0$ and Ishii and Souganidis \cite{IshiiSoug} for unbounded singularities only at $p=0$, Gurtin, Soner, and Souganidis \cite{gurtin_soner_souganidis} and Ohnuma and Sato \cite{ohnuma_sato} both for bounded singularities at $0$ and finitely many other directions. Since then, the first author extended in \cite{level_set} the idea of \cite{ohnuma_sato} to deal with countably many bounded discontinuities in the context of homogenization of level set PDE.  As discussed in Appendix \ref{Appendix: 2d}, that approach adapts more-or-less immediately to deal with the countable discontinuities of Finsler infinity Laplacian operators in dimension two.
\vs
Crandall, Gunnarsson, and Wang showed in \cite{crandall_gunnarsson_wang} how to prove comparison results for PDE involving certain Finsler infinity Laplacian operators via an approach similar to the one in \cite{ishii}.  That work anticipated but did not prove our results on polyhedral and two-dimensional norms.  In particular, our construction of shielding norms uses the same mollification trick originally suggested in \cite{crandall_gunnarsson_wang}.
\vs
Note that while the arguments in \cite{crandall_gunnarsson_wang} apply to Finsler infinity Laplacian operators, they do not seem to apply to an operator like the one in \eqref{E: median operator}.
\vs
To the best of our knowledge, \cite{crandall_gunnarsson_wang} is the only previous work providing a complete proof of comparison for an operator compatible with a polyhedral norm.  A number of earlier papers observed that such operators, usually the $\ell^{1}-$infinity Laplacian, arise as scaling limits of variational problems; see \cite{ishibashi_koike, belloni_kawohl, di-castro_perez-llanos_urbano, perez-llanos_rossi, rossi_saez}.

\subsection*{Organization of the paper}  In the next section, we review some preliminaries that are used throughout.  Section 3 is a warm up: we present the proof of Theorem \ref{T: main parabolic}, state an analogous result for elliptic problems and then discuss the $C^{2}$ and strictly convex settings, that is, Theorem \ref{T: C2 norms} and Proposition \ref{P: strictly convex norms}.  In Section 4, we prove the existence of shielding norms in the polyhedral setting.  Section 5 uses the previous results to establish comparison for the class of PDE from \cite{chatterjee_souganidis}.  Section 6 revisits questions pertaining to $L^{\infty}-$variational problems and related elliptic PDE.  
\vs
Finally, there are two appendices.  Appendix \ref{Appendix: 2d} describes some other comparison results that can be obtained in the simpler case when $d = 2$.  Appendix \ref{Appendix: technical lemmas} contains supplemental computations, the first being an explicit computation of the matrices involved in the definition of compatibility when $\varphi$ is the $\ell^{1}-$norm and the second, a proof that the operator $F_{1}$ of \eqref{E: median operator} cannot be rewritten as a Finsler infinity Laplacian.

\subsection*{Notation} We denote by $S^{d-1}$ the unit sphere in $\R^d$.  The number of points of  a finite set $A$ is $\sharp A$ and $\text{conv}(\{q_1,\ldots, q_N\})$ is the convex hull of $q_1,\ldots, q_N\in \R^d$. If $q \in \mathbb{R}^{d}$, we write $\langle q \rangle$ for its linear span, that is, \ $\langle q \rangle = \{\alpha q \, :\, \alpha \in \mathbb{R}\}$.
We write $\langle q, q' \rangle$ for the Euclidean inner product of two vectors $q,q' \in \mathbb{R}^{d}$, and  $\|\cdot\|$ is the  
Euclidean norm. Given a set $A \subseteq \mathbb{R}^{d}$, we denote by $A^{\perp}$ the set of vectors orthogonal to it, that is,
	\begin{equation*}
		A^{\perp} = \bigcap_{q' \in A} \{q \in \mathbb{R}^{d} \, :  \, \langle q, q' \rangle = 0\}.
	\end{equation*}

The standard orthonormal basis of $\mathbb{R}^{d}$ is denoted by $\{\bar{e}_{1},\bar{e}_{2},\dots,\bar{e}_{d}\}$.  We abbreviate the coordinates with respect to this basis by $p_{i} = \langle p, \bar{e}_{i} \rangle$.  Similarly, given $X \in \mathbb{R}^{d \times d}$, we write its matrix entries as $X_{ij} = \langle X \bar{e}_{j}, \bar{e}_{i} \rangle$.  
\vs

Given $V, U \subseteq \mathbb{R}^{d}$, we write $V \subset \subset U$ if the closure $\overline{V}$ of $V$ is a compact subset of $U$.
\vs

The space of symmetric matrices in $\mathbb{R}^{d \times d}$ is 
$\mathcal{S}^{d}$.  Given $X, Y \in \mathcal{S}^{d}$, we write $X \leq Y$ if $Y - X$ is positive semi-definite.
\vs
Given two vectors $q,q' \in \mathbb{R}^{d}$, the tensor product $q \otimes q'$ is the linear operator on $\mathbb{R}^{d}$ defined by 
	\begin{equation*}
		(q \otimes q') q'' = \langle q', q'' \rangle q.
	\end{equation*}

\section{Preliminaries}  

\subsection{Finsler norms} \label{S: Finsler norms}  We say that $\varphi : \mathbb{R}^{d} \to [0,\infty)$ is a Finsler norm if it is positively one-homogeneous, convex, and positive definite, that is, $\varphi$  satisfies the following three conditions:
	\begin{gather*}
		\varphi(\lambda q) = \lambda \varphi(q) \ \  \text{if} \ \ q \in \mathbb{R}^{d}, \ \ \lambda > 0,  \\
		\varphi(q_{1} + q_{2}) \leq \varphi(q_{1}) + \varphi(q_{2}) \ \  \text{if} \ \ q{_1}, q_{2} \in \mathbb{R}^{d}, \\
		\min \left\{ \frac{\varphi(q)}{\|q\|} \, : \, q \in \mathbb{R}^{d} \setminus \{0\} \right\} > 0.
	\end{gather*}	
We say that $\varphi$ is symmetric if $\varphi(-q) = \varphi(q)$ for all $q \in \mathbb{R}^{d}$, in which case it is a norm in the proper sense.
\vs	
Given a Finsler norm $\varphi$, we define its dual norm $\varphi^{*}$ by 
	\begin{equation*}
		\varphi^{*}(p) = \max \left\{ \frac{\langle p, q \rangle}{\varphi(q)} \, : \, q \in \mathbb{R}^{d} \setminus \{0\} \right\}.
	\end{equation*}
It is easy to check that $\varphi^{*}$ is also a Finsler norm.  Furthermore, $(\varphi^{*})^{*} = \varphi$.  
\vs

For each $q \in \mathbb{R}^{d}$, the subdifferential $\partial \varphi(q)$ of $\varphi$ at $q$ is defined by 
	\begin{equation*}
		\partial \varphi(q) = \left\{p \in \mathbb{R}^{d} \ :  \  \varphi(q') \geq \varphi(q) + \langle p, q' - q \rangle \ \text{for all} \  q' \in \mathbb{R}^{d}\right\}.
	\end{equation*}
We will frequently use the following representation of the subdifferential, which is specific to Finsler norms:
	\begin{equation} \label{E: subdifferential}
		\partial \varphi(q) = \left\{ \begin{array}{r l}
			\left\{ p \in \{\varphi^{*} = 1\} \, : \, \langle p, q \rangle = \varphi(q) \right\} & \text{if} \, \, q \in \mathbb{R}^{d} \setminus \{0\}, \\[1.5mm]
			\{\varphi^{*} = 1\}, & \text{if} \, \, q = 0.
		\end{array} \right.
	\end{equation}
Note that the same considerations apply to $\partial \varphi^{*}$.
\vs
A very useful fact about convex functions is they are almost $C^{1}$.  More precisely, the subdifferential $\partial \varphi$ is upper semi-continuous.  In fact, upper semi-continuity continues to hold if both the point and the norm are allowed to vary.  This is a classical fact. Since, however, we will need it later, we  state it carefully next.

\begin{prop} \label{P: upper semi-continuity}  Let $\varphi$ and $(\varphi_{n})_{n \in \mathbb{N}}$ be Finsler norms in $\mathbb{R}^{d}$ such that $\varphi_{n} \to \varphi$ locally uniformly as $n \to \infty$.  If $(q_{n})_{n \in \mathbb{N}}, (p_{n})_{n \in \mathbb{N}} \subseteq \mathbb{R}^{d}$ are sequences chosen such that $p_{n} \in \partial \varphi_{n}(q_{n})$ for each $n$ and if both sequences have limits $q = \underset{n \to \infty}\lim q_{n}$ and $p =\underset{n \to \infty} \lim p_{n}$, then $p \in \partial \varphi(q)$.  \end{prop}  

\begin{proof}  Given $q' \in \mathbb{R}^{d}$, we can write for any  $ n \in \mathbb{N}$, 
	\begin{equation*}
		\varphi_{n}(q') \geq \varphi_{n}(q_{n}) + \langle p_{n}, q' - q_{n} \rangle.
	\end{equation*}
Sending $n \to \infty$, this becomes
	\begin{equation*}
		\varphi(q') \geq \varphi(q) + \langle p, q' - q \rangle.
	\end{equation*}
Since $q'$ was arbitrary, this implies $p \in \partial \varphi(q)$ by definition. \end{proof}  

\subsection*{Terminology from Convex Analysis}  If $C \subseteq \mathbb{R}^{d}$ is a convex set, then there is a smallest linear subspace $V_{C} \subseteq \mathbb{R}^{d}$ such that
	\begin{equation*}
		C \subseteq q + V_{C} \ \  \text{if} \ \  q \in C.
	\end{equation*}
The relative interior of $C$, which we denote by $\text{ri}(C)$,  is, by definition, the interior of $C$ relative to the topology of $q + V_{C}$.
\vs

The boundary of $C$, denoted $\text{bdry} \, C$, is defined by $\text{bdry} \, C = \overline{C} \setminus \text{ri}(C)$, where $\overline{C}$ is the closure of $C$.
\vs 
The dimension of $C$ equals the dimension of the linear space $V_{C}$ defined above.
\vs
If $C \subseteq \mathbb{R}^{d}$ is convex, we say that a convex subset $F \subseteq C$ is a face of $C$ if, for each convex subset $F' \subseteq C$, 
	\begin{equation*}
		\text{if} \, \, \text{ri}(F') \cap F \neq \emptyset, \quad \text{then} \, \,  F' \subseteq F.  
	\end{equation*}
A point $x \in C$ is called an extreme point if $\{x\}$ is a face of $C$.

\subsection*{Polyhedral Finsler norms} \label{S: polyhedral norms} A Finsler norm $\varphi : \mathbb{R}^{d} \to [0,\infty)$ is called polyhedral if it is piecewise linear.  More precisely, $\varphi$ is polyhedral if there is a finite set of points $\{p_{1},\dots,p_{N}\} \subseteq \mathbb{R}^{d} \setminus \{0\}$ such that
	\begin{equation} \label{E: polyhedral norm}
		\varphi(q) = \max \left\{ \langle p_{1},q \rangle, \dots, \langle p_{N}, q \rangle \right\}.
	\end{equation}
There are a number of equivalent definitions.  Most notably, $\varphi$ is polyhedral if and only if the unit ball $\{\varphi \leq 1\}$ is a polytope.  
\vs
It is a classical fact the set of polyhedral Finsler norms is invariant under the dual operation. This is the topic of the next proposition. Its proof can be found in  Theorem 19.2 of Rockafellar \cite{rockafellar} or Section 2.4 of Schneider \cite{schneider}.

\begin{prop}\label{P: polyhedral norms duality} If $\varphi$ is a polyhedral Finsler norm, then $\varphi^{*}$ is also polyhedral.  \end{prop}  

The next proposition is about the fact that the points $\{p_{1},\dots,p_{N}\}$ in \eqref{E: polyhedral norm} can be canonically chosen. 

\begin{prop} \label{P: canonical choice} Let $\varphi$ be a polyhedral Finsler norm in $\mathbb{R}^{d}$.  If $E_{\varphi} \subseteq \{\varphi^{*} = 1\}$ is the set of extreme points of the dual ball $\{\varphi^{*} \leq 1\}$, then $\# E_{\varphi} < \infty$ and 
	\begin{equation*}
		\varphi(q) = \max \left\{ \langle p, q \rangle \, :  \, p \in E_{\varphi} \right\}.
	\end{equation*}      
\end{prop}

	\begin{proof}  Choose $\{p_{1},\dots,p_{N}\} \subseteq \mathbb{R}^{d} \setminus \{0\}$ so that $\varphi$ is given by \eqref{E: polyhedral norm}.  First, notice that \eqref{E: polyhedral norm} implies that $\{\varphi^{*} \leq 1\} = \text{conv}(\{p_{1},\dots,p_{N}\})$.  It is easy to verify that the set of extreme points of $\text{conv}(\{p_{1},\dots,p_{N}\})$ is a subset of $\{p_{1},\dots,p_{N}\}$.  In particular, $E_{\varphi} \subseteq \{p_{1},\dots,p_{N}\}$. Hence $\#E_{\varphi} \leq N < \infty$.
\vs 	
It is not hard to check directly that $\text{conv}(\{p_{1},\dots,p_{N}\})$ equals the convex hull of its extreme points.  This is true, in fact, for any compact convex set in $\mathbb{R}^{d}$, see \cite[Theorem 18.5]{rockafellar} or \cite[Theorem 1.4.3]{schneider}.  Thus, 
	\begin{equation*}
		\varphi(p) = \max \left\{ \langle p_{1},q \rangle, \dots, \langle p_{N}, q \rangle \right\} \leq \max \left\{ \langle p, q \rangle \, :  \, p \in E_{\varphi} \right\}.
	\end{equation*}
At the same time, since $E_{\varphi} \subseteq \{p_{1},\dots,p_{N}\}$, 
	\begin{equation*}
		\max \left\{ \langle p, q \rangle \, :  \, p \in E_{\varphi} \right\} \leq  \max \left\{ \langle p_{1},q \rangle, \dots, \langle p_{N}, q \rangle \right\} = \varphi(p).
	\end{equation*}
\end{proof}    
We close the general discussion about Finsler norms  with some examples.  Let $\varphi$ be the $\ell^{1}-$norm in $\mathbb{R}^{d}$, that is,
	\begin{equation*}
		\varphi(q) = \sum_{i = 1}^{d} |q_{i}|,
	\end{equation*}
which is polyhedral since we can write it in the form
	\begin{equation*}
		\varphi(q) = \max \Big\{ \sum_{i = 1}^{d} \rho(i) \langle q, e_{i} \rangle \, : \, \rho \in \{-1,1\}^{d} \Big\}.
	\end{equation*}
Its dual is readily seen to be the $\ell^{\infty}-$norm
	\begin{equation*}
		\varphi^{*}(p) = \max\left\{ |p_{1}|,\dots,|p_{d}| \right\}.
	\end{equation*}
	
Retaining the notation $\varphi$ for the $\ell^{1}-$norm, a less common example, which turns out to be of interest, is the norm $\underline{\varphi}$ defined by 
	\begin{equation*}
		\underline{\varphi}(q) = \max \{\varphi(q)\} \cup \left\{(d - 1) |q_{1}|,\dots, (d-1) |q_{d}| \right\}.
	\end{equation*}
The dual norm $\underline{\varphi}^{*}$ is given by 
	\begin{equation} \label{E: rhombic dodecahedron}
		\underline{\varphi}^{*}(p) = \max \Big\{ \dfrac{1}{d-1} \sum_{i \in \{1,2,\dots,d\} \setminus \{j\}} |p_{i}| \, : \, j \in \{1,2,\dots,d\} \Big\}.
	\end{equation}
When $d = 3$, the dual unit ball $\{\underline{\varphi}^{*} \leq 1\}$ is a rhombic dodecahedron.  
\vs
The norm $\underline{\varphi}$ appears in Section \ref{S: p laplace} in the parameter regime $\alpha < 2$.  In that context, the transformation $\varphi \mapsto \underline{\varphi}$ seems to be important.  It is not clear to us how to interpret the geometric meaning of this map or how it relates to the underlying discrete schemes, though.  
	
\subsection*{A fundamental matrix lemma}  We will need the following result, which is already used in \cite{ishii}.

	\begin{lemma} \label{L: matrix lemma} Given any linear subspace $V \subseteq \mathbb{R}^{d}$, if $0 \leq A \in \mathcal{S}_{V}$ and $X \in \mathcal{S}^{d}$ satisfies $-c A \leq X \leq cA$ for some $c > 0$, then $X \in \mathcal{S}_{V}$. \end{lemma}  
	
		\begin{proof}  Since the matrices in question are all symmetric, it suffices to prove that the range of $X$ is contained in $V$.
		\vs
		Fix $u \in \mathbb{R}^{d}$.  Given any $w \in V^{\perp}$, the parallelogram identity gives
			\begin{align*}
				4 \langle Xu, w \rangle &= \langle X(u + w), u + w \rangle - \langle X(u - w), u - w \rangle \\
				&\leq c \langle A (u + w), u + w \rangle - (-c\langle A(u - w), u - w \rangle) \\
				&= 2 c \langle Au, u \rangle.
			\end{align*}
		Thus, the linear functional $w \mapsto \langle Xu, w \rangle$ is bounded above on $V^{\perp}$.  It follows easily that $Xu \in (V^{\perp})^{\perp} = V$.       \end{proof}  
		
\subsection*{The Finsler infinity Laplacian} \label{S: infinity laplace}  Throughout the paper, if $\varphi$ is any Finsler norm, its infinity Laplacian is the multi-valued operator
	\begin{equation*}
		\langle D^{2} u \cdot \partial \varphi^{*}(Du), \partial \varphi^{*}(Du) \rangle = \left\{\langle D^{2}u \cdot q, q \rangle \, : \, q \in \partial \varphi^{*}(p) \right\}.
	\end{equation*}
For a discussion of the role of these nonlinearities in the theory of $L^{\infty}-$variational problems, see Section 5 of the survey by Aronsson, Crandall, and Juutinen \cite{aronsson_crandall_juutinen} or the paper by Armstrong, Crandall, Julin, and Smart \cite{armstrong_crandall_julin_smart}.  Here we use viscosity solutions tools to study these operators.
\vs
Accordingly, define a pair of operators $(\overline{G}_{\varphi},\underline{G}^{\varphi})$ by 
	\begin{equation} \label{E: infinity laplace}
		\left\{ \begin{array}{c} 
			\overline{G}_{\varphi}(p,X) = \max \left\{ \langle Xq, q \rangle \, :  \, q \in \partial \varphi^{*}(p) \right\}, \\[1.5mm]
			\underline{G}^{\varphi}(p,X) = \min \left\{ \langle Xq, q \rangle \, :  \, q \in \partial \varphi^{*}(p) \right\}.
		\end{array} \right.
	\end{equation}
The next  proposition asserts that the pair $(\overline{G}_{\varphi},\underline{G}^{\varphi})$ satisfies our main assumptions.

\begin{prop} \label{P: infinity laplace}  Given any Finsler norm $\varphi$ in $\mathbb{R}^{d}$, the pair $(\overline{G}_{\varphi},\underline{G}^{\varphi})$ satisfies \eqref{A: elliptic_1}, \eqref{A: elliptic_2}, and \eqref{A: semicontinuous} and it is compatible with $\varphi$.  \end{prop}

	\begin{proof}  It is clear that the definitions imply  \eqref{A: elliptic_1} and \eqref{A: elliptic_2} while  \eqref{A: semicontinuous}  is a consequence of the upper semi-continuity of the subdifferential $\partial \varphi^{*}$ (see Proposition \ref{P: upper semi-continuity}).
\vs	
	To prove compatibility, that is,  \eqref{A: consistency}, suppose that $p \in \mathbb{R}^{d}$ and $X \in \mathcal{S}(p,\varphi)$.  If $p = 0$, then $\mathcal{S}(p,\varphi) = \{0\}$ and it is clear that $\overline{G}_{\varphi}(0,0) = 0 = \underline{G}^{\varphi}(0,0)$.  Therefore, we can assume $p \neq 0$.  
\vs	
	To show that $\overline{G}_{\varphi}(p,X) = \underline{G}^{\varphi}(p,X)$, notice that it suffices to prove that the quadratic form $v \mapsto \langle Xv, v \rangle$ is constant in $\partial \varphi^{*}(p)$.  Here it is easiest to appeal to linearity.  Observe that $\mathcal{S}(p,\varphi)$ is spanned by the following set of elementary tensors
		\begin{equation*}
			\{p \otimes p\} \cup \{v \otimes p + p \otimes v \, :  \, v \in \mathcal{T}(p,\varphi)\} \cup \{v \otimes v' + v' \otimes v \, :  \, v,v' \in \mathcal{T}(p,\varphi)\}.
		\end{equation*}
	Therefore, by linearity, we only need to check that $v \mapsto \langle Xv, v \rangle$ is constant in $\partial \varphi^{*}(p)$ when $X$ is equal to one of these tensors.
\vs	
	In case $X = p \otimes p$, the representation formula \eqref{E: subdifferential}  implies that
		\begin{equation*}
			\langle p,q \rangle = \varphi^{*}(p)\ \ \text{if} \ \ q \in \partial \varphi^{*}(p),
		\end{equation*}
and,	thus, $\overline{G}_{\varphi}(p,X) = \underline{G}^{\varphi}(p,X) = \varphi^{*}(p)^{2}$.
\vs	
	When $X = p \otimes v + v \otimes p$ for some $v \in \mathcal{T}(p,\varphi)$, we know that $\partial \varphi^{*}(p) \subseteq \mathcal{T}(p,\varphi)^{\perp}$ by the definition of $\mathcal{T}(p,\varphi)$.  Therefore,
		\begin{equation*}
			\langle v, q \rangle = 0 \ \  \text{if} \ \ q \in \partial \varphi^{*}(p),
		\end{equation*}
	From this, we deduce that
		\begin{equation*}
			\langle (p \otimes v + v \otimes p) e, e \rangle = 2\langle p,e \rangle \langle v,e \rangle = 0 \quad \text{if} \, \, q \in \partial \varphi^{*}(p),
		\end{equation*}
and,	hence $\overline{G}_{\varphi}(p,X) = \underline{G}^{\varphi}(p,X) = 0$.
\vs	
	Finally, if $X = v \otimes v' + v' \otimes v$ for some $v,v' \in \mathcal{T}(p,\varphi)$, then the arguments of the previous paragraph again yield $\overline{G}_{\varphi}(p,X) = \underline{G}^{\varphi}(p,X) = 0$.    \end{proof}  

While at a technical level, we will use the pair $(\overline{G}_{\varphi},\underline{G}^{\varphi})$ to describe differential inequalities involving the $\varphi$-infinity Laplacian, when there is no risk of confusion, we will write  
	\be\label{E: phi infinity harmonic} 
	\begin{split}
		&- \langle D^{2}u \cdot \partial \varphi^{*}(Du), \partial \varphi^{*}(Du) \rangle = 0 \ \ \text{in} \ \ \Omega,  \\[1.2mm]
		&- \langle D^{2} w \cdot \partial \varphi^{*}(Dw), \partial \varphi^{*}(Dw) \rangle \leq 0 \ \ \text{in} \ \  \Omega,\\[1.2mm]  &- \langle D^{2}v \cdot \partial \varphi^{*}(Dv), \partial \varphi^{*}(Dv) \rangle \geq 0 \ \  \text{in} \ \ \Omega, 
	\end{split}
	\ee
in place of the expressions 
	\begin{gather*}
		-\overline{G}_{\varphi}(Du,D^{2}u) \leq 0 \ \ \text{and} \ \ -\underline{G}^{\varphi}(Du,D^{2}u) \geq 0 \ \ \text{in} \ \ \Omega, \\[1.2mm] 
		-\overline{G}_{\varphi}(Dw,D^{2}w) \leq 0 \ \ \text{and} \ \ - \underline{G}^{\varphi}(Dv,D^{2}v) \geq 0 \ \ \text{in} \ \ \Omega,
	\end{gather*}
with inequalities and equalities to be understood in the viscosity sense.
\vs
As a shorthand, we will say that $u$ is $\varphi$-infinity harmonic in $\Omega$ when \eqref{E: phi infinity harmonic} holds.

\subsection*{Median} \label{S: median}  Given a finite set $A = \{a_{0},\dots,a_{N}\}$ of real numbers $a_{0} \leq a_{1} \leq \dots \leq a_{N}$, we define the median $\text{Med}(A)$ by 
	\begin{equation*}
		\text{Med}(A) = \left\{ \begin{array}{r l}
								a_{N/2}, & \text{if} \, \, N \in 2 \mathbb{Z}, \\
								\frac{1}{2} \left(a_{(N-1)/2} + a_{(N+1)/2} \right), & \text{otherwise.}
						\end{array} \right.
	\end{equation*}
In connection with the schemes of Section \ref{S: p laplace}, it is worth noting that $\text{Med}(A)$ is a minimizer of the problem
	\begin{equation*}
		\min \Big\{ \sum_{a \in A} |y - a| \, :  \, y \in \mathbb{R} \Big\},
	\end{equation*}
which is not unique when $N \in 2 \mathbb{Z} + 1$.

\section{Ishii's Method and the $C^{2}-$Case} \label{S: warm up}

We start by recalling the proof of Theorem \ref{T: main parabolic} in order to demonstrate the role played by each of the assumptions on the norm $\varphi$ and the operators $(\overline{F},\underline{F})$.  We also state without proof an elliptic analogue.
\vs 
The rest of the section is devoted to a discussion of the cases  $\varphi^{*} \in C^{1}(\mathbb{R}^{d} \setminus \{0\})$ or $\varphi \in C^{2}(\mathbb{R}^{d} \setminus \{0\})$.  We start by observing that in the former  the associated equations are continuous away from $\{0\} \times \mathcal{S}^{d}$ and the comparison follows by a now classical approach.  In the $C^{2}-$case, things are more interesting, and it turns out that the discussion is related to the theory of infinity harmonic functions.
\vs
We end the section discussing an example from \cite{ishii}.  In particular, we show how the assumptions of \cite{ishii} fit into the framework of this paper.

\subsection*{Proof of Theorem \ref{T: main parabolic}}  This proof appears in  \cite{ishii}.  We recall it here to keep the presentation self-contained and to make explicit the role of the compatibility assumption \eqref{A: consistency} and the definition of shielding norm, that is,  Definition \ref{D: good test function}.

\begin{proof}[Proof of Theorem \ref{T: main parabolic}]  We argue by contradiction.  If the statement of the theorem were false, we could find a $\sigma > 0$ such that
	\begin{equation*}
		\sup \left\{ w(x,t) - v(x,t) - \sigma t \, :  \, (x,t) \in \mathbb{R}^{d} \times [0,T) \right\} > 0.
	\end{equation*}

As usual, we double variables.  Fix $\zeta, \beta > 0$ and consider the function $\Phi = \Phi_{\zeta,\beta} : \mathbb{R}^{d} \times \mathbb{R}^{d} \times [0,T] \to \mathbb{R}$ given by
	\begin{equation*}
		\Phi(x,y,t) = w(x,t) - v(y,t) - \frac{\psi(x - y)^{4}}{4 \zeta} - \beta \|y\|^{4} - \sigma t.
	\end{equation*}
Since $w$ and $v$ are bounded, $\Phi$ attains its maximum at some point $(\bar{x}_{\zeta,\beta},\bar{y}_{\zeta,\beta},\bar{t}_{\zeta,\beta})$, and we have, suppressing the dependence on $\beta$ and $\zeta$,
	\begin{equation*}
		\sup \Big\{ \frac{\|\bar{x} -\bar{y}\|^{4}}{\zeta} + \beta \|\bar{y}\|^{4} \, :  \, (\zeta,\beta) \in (0,\infty)^{2} \Big\} < \infty.
	\end{equation*}
Furthermore, in view of the assumptions on $w(\cdot,0)$ and $v(\cdot,0)$, it is clear that there are constants $\zeta_{0}, \beta_{0} > 0$ such that $\bar{t}_{\zeta,\beta} > 0$ as soon as $(\zeta,\beta) \in (0,\zeta_{0}) \times (0,\beta_{0})$.  Henceforth, we choose $(\zeta, \beta)$ accordingly.
	\vs
In what follows, define $\bar{A} : \mathbb{R}^{d} \to \mathcal{S}^{d}$ and $\bar{p} : \mathbb{R}^{d} \to \mathbb{R}^{d}$ by 
\[		\bar{A}(\xi) = 3 \zeta^{-1} \psi(\xi)^{2} D\psi(\xi) \otimes D\psi(\xi) + \zeta^{-1} \psi(\xi)^{3} D^{2}\psi(\xi) \ \ \text{and} \ \ \bar{p}(\xi) = \zeta^{-1} \psi(\xi)^{3} D\psi(\xi),\]
together with the interpretation that $\bar{p}(0) = 0$ and $\bar{A}(0) = 0$.  Since $\psi$ is a shielding norm for $\varphi$, the inclusion $\bar{A}(\xi) \in \mathcal{S}(\bar{p}(\xi),\varphi)$ holds for each $\xi \in \mathbb{R}^{d}$.
\vs 	
Applying \cite[Lemma 1]{ishii},  which is a variant of the maximum principle for semi-continuous functions, we find matrices $X = X_{\zeta,\beta}, Y = Y_{\zeta,\beta} \in \mathcal{S}^{d}$ and real numbers $a,  b \in \mathbb{R}$ such that 
	\begin{gather*}
		\sigma = a - b, \quad a - \overline{F}(\bar{p}(\bar{x} - \bar{y}), X) \leq 0, \\
		b - \underline{F}(\bar{p}(\bar{x} - \bar{y}) - 4 \beta \|\bar{y}\|^{2} \bar{y}, Y - 8 \beta \bar{y} \otimes \bar{y} - 4 \beta \|\bar{y}\|^{2} \text{Id} ) \geq 0, \\
		-3 \left( \begin{array}{c c}
				\bar{A}(\bar{x} - \bar{y}) & 0 \\
				0 & \bar{A}(\bar{x} - \bar{y})
			\end{array} \right) 
			\leq \left( \begin{array}{c c}
									X & 0 \\
									0 & -Y
								\end{array} \right) 
								\leq 3 \left( \begin{array}{c c}
														\bar{A}(\bar{x} - \bar{y}) & -\bar{A}(\bar{x} - \bar{y}) \\
														-\bar{A}(\bar{x} - \bar{y}) & \bar{A}(\bar{x} - \bar{y}) 
														\end{array} \right).
	\end{gather*}
Note, in particular, that $-3\bar{A}(\bar{x} - \bar{y}) \leq X \leq Y \leq 3 \bar{A}(\bar{x} - \bar{y})$.
\vs 
At this stage, we send $\beta \to 0^{+}$.  In view of the matrix inequalities satisfied by $X$ and $Y$ and the bound on $\zeta^{-1} \|\bar{x} - \bar{y}\|^{4}$, we can fix $\bar{\xi} \in \mathbb{R}^{d}$ and $\bar{X}, \bar{Y} \in \mathcal{S}^{d}$ such that, up to subsequences,
	\begin{equation*}
		\bar{\xi} = \lim_{\beta \to 0^{+}} (\bar{x}_{\zeta, \beta} - \bar{y}_{\zeta,\beta}), \quad \bar{X} = \lim_{\beta \to 0^{+}} X_{\zeta,\beta}, \quad \bar{Y} = \lim_{\beta \to 0^{+}} Y_{\zeta,\beta}.
	\end{equation*}
Further, by continuity, $-3 \bar{A}(\bar{\xi}) \leq \bar{X} \leq \bar{Y} \leq 3 \bar{A}(\bar{\xi})$, and  Lemma \ref{L: matrix lemma} gives $\bar{X}, \bar{Y} \in \mathcal{S}(\bar{p}(\bar{\xi}),\varphi)$.
\vs
Using the bounds on $(\bar{y}_{\zeta,\beta})_{\beta > 0}$, we can eliminate the error terms involving $\beta$ in the previous inequalities to get
	\begin{equation*}
		\sigma + \underline{F}(\bar{p}(\bar{\xi}), \bar{Y}) - \overline{F}(\bar{p}(\bar{\xi}), \bar{X}) \leq 0.
	\end{equation*}
Invoking the inclusion $\bar{X}, \bar{Y} \in \mathcal{S}(\bar{p}(\bar{\xi}),\varphi)$, we find $\overline{F}(\bar{p}(\bar{\xi}),\bar{X}) = \underline{F}(\bar{p}(\bar{\xi}),\bar{X})$.  Thus, since $\bar{X} \leq \bar{Y}$ and $\underline{F}$ is elliptic, that is, \eqref{A: elliptic_2} holds, we arrive at the contradiction
	\begin{equation*}
		0 < \sigma \leq \sigma + \underline{F}(\bar{p}(\bar{\xi}), \bar{Y}) - \overline{F}(\bar{p}(\bar{\xi}), \bar{X}) \leq 0.
	\end{equation*}
\end{proof}    	

\subsection*{Comparison for elliptic problems}  As we already acknowledged in the introduction, there is no need to restrict attention to parabolic problems.  Here is a prototypical elliptic variant of Theorem \ref{T: main parabolic} that will be useful later in the paper.

\begin{prop} \label{P: main elliptic}  Suppose that $\varphi$ is a Finsler norm in $\mathbb{R}^{d}$, $(\overline{F},\underline{F})$ is a pair of operators satisfying \eqref{A: elliptic_1}, \eqref{A: elliptic_2}, and \eqref{A: semicontinuous} that are compatible with the geometry of $\varphi$, and $\varphi$ possesses a shielding norm $\psi$. 
\vs
If $\Omega \subseteq \mathbb{R}^{d}$ is a bounded open set, $f : \Omega \to (0,\infty)$ is a positive continuous function, and $(w,v) \in USC(\Omega) \times LSC(\Omega)$ satisfy
	\begin{gather*}
		- \overline{F}(Dw,D^{2}w) \leq 0 \ \ \text{in} \ \ \Omega, \quad -\underline{F}(Dv,D^{2}v) \geq f \ \ \text{in} \ \ \Omega, \\
		\lim_{\delta \to 0^{+}} \sup \left\{ w(x) - v(y) \, :  \, \|x - y\| + \text{dist}(x,\partial \Omega) + \text{dist}(y,\partial \Omega) \leq \delta \right\} \leq 0,
	\end{gather*}
then 
	\begin{equation*}
		\sup \left\{ w(x) - v(x) \, :  \, x \in \Omega \right\} \leq 0.
	\end{equation*}
\end{prop}   

The proof of Proposition \ref{P: main elliptic} is similar to that of Theorem \ref{T: main parabolic}, hence we omit it.

\subsection*{Strictly convex Finsler norms}  Recall that one of our main examples is the Finsler infinity Laplacian or the pair $(\overline{G}_{\varphi},\underline{G}^{\varphi})$ of \eqref{E: infinity laplace}.  It is clear that, if $\varphi^{*} \in C^{1}(\mathbb{R}^{d} \setminus \{0\})$ or, equivalently, $\{\varphi \leq 1\}$ is strictly convex, then $\overline{G}_{\varphi} \equiv \underline{G}^{\varphi}$ outside of $\{0\} \times \mathcal{S}^{d}$ so we are working with a continuous nonlinearity as long as the gradient does not vanish.  It is natural to ask if this is always necessarily the case for nonlinearities  compatible with such a Finsler norm.
Indeed, this is exactly the content of Proposition \ref{P: strictly convex norms}.
\vs
Notice that the final statement of the proposition implies that the $1$-Laplacian operator $F(p,X) = \text{tr}((\text{Id} - \|p\|^{-2} p \otimes p)X)$ is associated to a pair $(\overline{F},\underline{F})$ compatible with the Euclidean norm.  Hence, as already observed in \cite{ishii}, the comparison principle for the level set mean curvature flow can be seen as a special case of Theorem \ref{T: main parabolic}.

	\begin{proof}[Proof of Proposition \ref{P: strictly convex norms}]  Since $\varphi^{*} \in C^{1}(\mathbb{R}^{d} \setminus \{0\})$, we have $\partial \varphi^{*}(p) = \{D\varphi^{*}(p)\}$ for each $p \in \mathbb{R}^{d} \setminus \{0\}$.  Accordingly, since the direct sum of a line and a $(d - 1)$-dimensional linear space has dimension $d$, we find
		\begin{equation*}
			\langle p \rangle \oplus \mathcal{T}(p,\varphi) = \langle p \rangle \oplus \langle D\varphi^{*}(p) \rangle^{\perp} = \mathbb{R}^{d} \ \  \text{if} \ \  p \in \mathbb{R}^{d} \setminus \{0\}.
		\end{equation*}
Thus, by definition, $\mathcal{S}(\cdot,\varphi) \equiv \mathcal{S}^{d}$ in $\mathbb{R}^{d} \setminus \{0\}$, and the conclusions of the proposition follow immediately.   \end{proof}

\subsection*{$C^{2}-$norms}  We now prove Theorem \ref{T: C2 norms}, that is, we show that, if $\varphi \in C^{2}(\mathbb{R}^{d} \setminus \{0\})$, then $\varphi$ is a shielding norm for itself.  This turns out to be intimately related to the fact that any Finsler norm is a solution of its own infinity harmonic equation, a fact which is the topic of Lemma \ref{L: infinity harmonic cones} below.

	\begin{prop} \label{P: smooth case}  For each $q \in \mathbb{R}^{d} \setminus \{0\}$, 
		\begin{equation*}
			\partial \varphi^{*}(D\varphi(q)) \subseteq \text{Ker} \, D^{2}\varphi(q).
		\end{equation*}
	\end{prop}

Notice that the proposition implies that $\varphi$ is $\varphi$-infinity harmonic in $\mathbb{R}^{d} \setminus \{0\}$.  Unlike some known examples of smooth infinity harmonic functions, though, $\varphi$ has the property that its $\varphi$-infinity Laplacian is unambiguously defined.  More precisely, Proposition \ref{P: smooth case} implies that
	\begin{equation*}
		-\overline{G}_{\varphi}(D\varphi(x),D^{2}\varphi(x)) = -\underline{G}^{\varphi}(D\varphi(x),D^{2}\varphi(x)) = 0 \ \  \text{if} \ \ x \in \mathbb{R}^{d} \setminus \{0\}.
	\end{equation*}
For an example of a smooth $\ell^{1}-$infinity harmonic function for which this identity fails, see \cite[Example 5.2]{aronsson_crandall_juutinen}.
\vs
Before proving the proposition, we show how it implies that $C^{2}-$norms are shielding norms.

\begin{proof}[Proof of Theorem \ref{T: C2 norms}]    Fix $q \in \mathbb{R}^{d} \setminus \{0\}$.  To see that $\varphi$ is a shielding norm, we need to show that $D^{2}\varphi(q) \in \mathcal{S}(D\varphi(q),\varphi)$.  By Proposition \ref{P: smooth case}, we know that $\partial \varphi^{*}(D\varphi(q)) \subseteq \text{Ker} \, D^{2}\varphi(q)$.  Thus, employing the notation of Section \ref{S: definition of compatibility}, we have $D^{2}\varphi(q) \in \mathcal{S}_{V}$ with $V = \partial \varphi^{*}(D\varphi(q))^{\perp} = \mathcal{T}(D\varphi(q),\varphi)$.  This implies the desired inclusion since $\mathcal{S}_{V} \subseteq \mathcal{S}(D\varphi(q),\varphi)$.
 \end{proof}  

	\begin{proof}[Proof of Proposition \ref{P: smooth case}]  By homogeneity, we can assume that $\varphi(q) = 1$.  We also note that $\partial \varphi^{*}(D\varphi(q))$ is a boundary face of $\{\varphi \leq 1\}$ containing $q$.  To begin with, fix $\tilde{q} \in \text{ri}(\partial \varphi^{*}(D\varphi(q)))$.  It is not hard to show that $q' \mapsto D\varphi(q')$ is a constant map in $\text{ri}(\partial \varphi^{*}(D\varphi(q))$.  Thus, if $q' \in \partial \varphi^{*}(D\varphi(q))$, it follows that, for $t \in (0,1)$ small enough,
		\begin{equation*}
			D\varphi(\tilde{q} + t(q' - \tilde{q})) = D\varphi(\tilde{q}).
		\end{equation*}
	Differentiating the equation above yields  $D^{2} \varphi (\tilde{q}) (q' - \tilde{q}) = 0$.  On the other hand, $D^{2} \varphi(\tilde{q})\tilde{q} = 0$ since $\varphi$ is linear along the line $\langle \tilde{q} \rangle$.  We conclude that $D^{2}\varphi(\tilde{q}) q' = 0$.
	\vs
	The previous paragraph completes the argument when $q \in \text{ri}(\partial \varphi^{*}(D\varphi(q))$.  If $q \in \text{bdry} \, \partial \varphi^{*}(D\varphi(q))$, then we can find a sequence $(\tilde{q}_{n})_{n \in \mathbb{N}}$ as above such that $\tilde{q}_{n} \to q$.  Since $D^{2} \varphi$ is continuous, we conclude that $D^{2}\varphi(q) q' = \lim_{n \to \infty} D^{2} \varphi(\tilde{q}_{n})q' = 0$ for all $q' \in \partial \varphi^{*}(D\varphi(q))$. \end{proof}

\subsection*{Example: Anisotropic Curvature Flows} \label{Example: anisotropic curvature flows}  One of the main examples of \cite{ishii} is a class of anisotropic curvature flows.  One strategy for studying the flow with respect to a given a Finsler norm $\psi$ is to consider the level set PDE
	\begin{equation} \label{E: curvature}
		u_{t} - \text{tr} \left( D^{2} \psi(\widehat{Du}) D^{2}u \right) = 0 \ \  \text{in} \ \ \mathbb{R}^{d} \times (0,\infty).
	\end{equation}
If $\psi \notin C^{2}$, then the coefficients are not continuous so the usual comparison principle for mean curvature flow does not apply.  
\vs 
Nonetheless, as pointed out in \cite{ishii}, it is possible to extend the comparison principle to apply to certain piecewise $C^{2}-$Finsler norms.  
\vs 
The idea is the following. Let $O$ be a $C^{2}$ open subset of $S^{d-1}$ and  $M = \partial O$, and assume that $\psi_{1}, \psi_{2} \in C^{2}(\mathbb{R}^{d} \setminus \{0\})$ are Finsler norms such that
	\begin{equation} \label{E: first order identity}
		\psi_{1}(p) = \psi_{2}(p), \quad D\psi_{1}(p) = D\psi_{2}(p) \ \  \text{if} \ \ p \in \mathbb{R}_{+} M \setminus \{0\}.
	\end{equation}
If $\psi$ is given by 
	\begin{equation*}
		\psi(p) = \left\{ \begin{array}{r l}
						\psi_{1}(p) & \text{if} \, \, p \in \mathbb{R}_{+}O, \\
						\psi_{2}(p) & \text{otherwise},
					\end{array} \right.
	\end{equation*}
then $D^{2}\psi$ may not be well-defined on $\mathbb{R}_{+} M$.  
\vs
Even so, it is still possible to make sense of \eqref{E: curvature}.  Notice that, for each $p \in \mathbb{R}_{+} M \setminus \{0\}$, the directional second derivative $\langle D^{2}\psi(p) \eta,\eta \rangle$ is well-defined for each $\eta \in T_{p}M$.  This follows from \eqref{E: first order identity} and some calculation.  Therefore, if $\overline{F}$ and $\underline{F}$ are as follows 
	\begin{align*}
		\overline{F}(p,X) &= \lim_{\delta \to 0^{+}} \sup \Big\{ \text{tr} (D^{2}\psi(\widehat{p'}) X) \, :  \, \|p' - p\| < \delta \Big\}, \\ 
		\underline{F}(p,X) &= \lim_{\delta \to 0^{+}} \inf \Big\{ \text{tr} (D^{2}\psi(\widehat{p'}) X) \, :  \, \|p' - p\| < \delta \Big\},
	\end{align*}
and if $\mathcal{S}(p,M) = \mathcal{S}_{\langle p \rangle \oplus T_{p}M}$ for $p \in \mathbb{R}_{+} M$ and $\mathcal{S}(0,M) = \{0\}$, then 
	\begin{align*}
		\overline{F}(p,X) &= \underline{F}(p,X) \ \  \text{if} \, \, p \in \mathbb{R}_{+}M, \ \ X \in \mathcal{S}(p,M), \\[1.2mm]
		\underline{F}(p,X) &= \overline{F}(p,X) \quad \text{if} \, \, (p,X) \in (\mathbb{R}^{d} \setminus \mathbb{R}_{+}M) \times \mathcal{S}^{d}.
	\end{align*}
These last identities should remind the reader of \eqref{A: consistency}.  Indeed, one of the main technical devices of \cite{ishii} is the following result:

	\begin{prop}[\cite{ishii}, Lemma 1]  If $\mathcal{O}$, $\overline{F}$, and $\underline{F}$ are as in the preceding discusison, then there is a Finsler norm $\varphi \in C^{2}(\mathbb{R}^{d} \setminus \{0\})$ such that, for each $p \in \mathbb{R}_{+} M$,
		\begin{equation*}
			\mathcal{S}(p,\varphi) = \mathcal{S}(p,M).
		\end{equation*}
	In particular, the pair $(\overline{F},\underline{F})$ is compatible with $\varphi$.  \end{prop}  	
	
\section{Polyhedral Norms} \label{S: polyhedral norm}

We now proceed to the polyhedral setting.  The next result is the main technical ingredient that makes the study possible.  Roughly, it says that, while $\varphi$ is far from being $C^{2}$, we can approximate it by $C^{2}$ $\varphi$-infinity harmonic functions.  Hence even though $\varphi$ is not $C^{2}$, we are not too far from Proposition \ref{P: smooth case}.

\begin{prop} \label{P: polyhedral case}  Assume that $\varphi : \mathbb{R}^{d} \to [0,\infty)$ is a polyhedral Finsler norm.

(a) For each $c > 0$, there is a sequence of smooth, convex functions $(f_{c,n})_{n \in \mathbb{N}}$ such that $f_{c,n} \to \varphi$ uniformly in $\{\varphi \geq c\}$ as $n \to \infty$ and
	\begin{equation} \label{E: infinity harmonic}
		\varphi^{*}(Df_{c,n}(q)) = 1 \quad \text{and} \quad D^{2} f_{c,n}(q) \in \mathcal{S}_{\mathcal{T}(Df_{c,n}(q))} \ \ \text{if} \ \ q \in \{\varphi \geq c\}.
	\end{equation}
In particular,
	\begin{equation}
		\partial \varphi^{*}(Df_{c,n}(q)) \subseteq \text{Ker} \, D^{2}f_{c,n}(q) \ \  \text{if} \ \ q \in \{\varphi \geq c\}.
	\end{equation}

(b) There is a sequence of Finsler norms $(\varphi_{n})_{n \in \mathbb{N}} \subseteq C^{2}(\mathbb{R}^{d} \setminus \{0\})$ such that, as $n \to \infty$, $\varphi_{n} \to \varphi$ locally uniformly in $\mathbb{R}^{d}$ and, for each $q \in \mathbb{R}^{d} \setminus \{0\}$ and $n \in \mathbb{N}$,
	\begin{equation*}
		D^{2}\varphi_{n}(q) \in \mathcal{S}(D\varphi_{n}(q)).
	\end{equation*}
In particular, $\varphi$ possesses a shielding norm.
\end{prop}  

Notice that part (b) of the proposition implies Theorem \ref{T: polyhedral norms}.  While we are mainly interested in these shielding norms, the convex functions of part (a) are of independent interest.  
\vs
In the next part of the current section, we show how the convex functions of part (a) can be used to provide a quick proof that the cone comparison properties of the $L^{\infty}-$ variational theory are equivalent to the viscosity $\varphi$-subharmonic and $\varphi$-superharmonic properties in this context.  The remainder of the section is devoted to the proof of Proposition \ref{P: polyhedral case}.  We begin by constructing the functions $(f_{c,n})_{n \in \mathbb{N}}$ and then use these to obtain $(\varphi_{n})_{n \in \mathbb{N}}$.

\subsection*{Application of part (a): Cone comparison} \label{S: gimmick cone comparison}  The Dirichlet problem for the Finsler infinity Laplacian, which is one of the basic problems of the theory of $L^{\infty}-$\\ 
variational problems, is the PDE
	\begin{equation} \label{E: infinity harmonic equation}
- \langle D^{2} u \cdot \partial \varphi^{*}(Du), \partial \varphi^{*}(Du) \rangle = 0 \ \  \text{in}  \ \   \Omega \ \ \text{and} \ \ u = g \ \ 
 \text{on} \ \ \partial \Omega.
\end{equation}	
In  \cite{armstrong_smart_easy_proof} Armstrong and Smart showed that it is relatively easy to prove comparison results for variational sub- and super-solutions of \eqref{E: infinity harmonic equation}.  As pointed out in \cite{armstrong_crandall_julin_smart}, this begs the question whether or not viscosity sub- and super-solutions are equivalent to the variational ones.  
\vs
We recall one definition of variational sub- and super-solutions, which is based on ``comparison with cones."     

\begin{definition}  Given an open set $\Omega \subseteq \mathbb{R}^{d}$ and a Finsler norm $\varphi$ in $\mathbb{R}^{d}$, we say that $w \in \text{USC}(\Omega)$ has the cone comparison property from above in $\Omega$, denoted $w \in CCA_{\varphi}(\Omega)$, if, for each $V \subset \subset \Omega$ open, $x_{0} \in \mathbb{R}^{d} \setminus V$, and $a > 0$, 
	\begin{equation*}
		\max \left\{ w(x) - a \varphi(x - x_{0}) \, :  \, x \in V \right\} = \max \left\{ w(x) - a \varphi(x - x_{0}) \, :  \, x \in \partial V \right\}.
	\end{equation*}
We say $v \in \text{LSC}(\Omega)$ has the cone comparison property from below in $\Omega$, denoted $v \in CCB_{\varphi}(\Omega)$, if $w = -v \in CCA_{\varphi}(\Omega)$\end{definition}

Next we  show that,  using the  smooth $\varphi$-infinity harmonic approximations of $\varphi$ obtained in Proposition \ref{P: polyhedral case}, it is easy to deduce that $CCA_{\varphi}(\Omega)$ and $CCB_{\varphi}(\Omega)$ coincide with the $\varphi$-infinity sub- and $\varphi$-infinity super-harmonic functions.

\begin{prop} \label{P: easy cones} If $\varphi$ is a polyhedral Finsler norm in $\mathbb{R}^{d}$, then
	\begin{itemize}
		\item[(i)] $w \in CCA_{\varphi}(\Omega)$ if and only if $-\langle D^{2} w \cdot \partial \varphi^{*}(Dw), \partial \varphi^{*}(Dw) \rangle \leq 0$ in $\Omega$.
		\vskip.05in
		\item[(ii)] $v \in CCB_{\varphi}(\Omega)$ if and only if $-\langle D^{2} v \cdot \partial \varphi^{*}(Dv), \partial \varphi^{*}(Dv) \rangle \geq 0$ in $\Omega$.
	\end{itemize} 
\end{prop}  

	\begin{proof}  The ``only if" direction is classical; see, for example, \cite{aronsson_crandall_juutinen} or \cite[Theorem 4.8]{armstrong_crandall_julin_smart}.
\vs	
We only prove  the ``if" direction for  (i), since (ii) follows by applying (i) to $-v$.

\vs	
	We argue by contradiction.  Assume that $w$ is $\varphi$-infinity subharmonic in $\Omega$ and yet there is a $V \subset \subset \Omega$, an $x_{0} \in \mathbb{R}^{d} \setminus V$, and an $a > 0$ such that 
		\begin{equation} \label{E: contradiction time}
			\max \left\{ w(x) - a \varphi(x - x_{0}) \, :  \, x \in V \right\} > \max \left\{ w(y) - a \varphi(y - x_{0}) \, :  \, y \in \partial V \right\}.
		\end{equation}
	
	To start with, observe that there is a $c > 0$ such that $\varphi(x - x_{0}) \geq 2c$ for all $x \in V$.  
	
	Let $(f_{c,n})_{n \in \mathbb{N}}$ be the smooth convex functions of Proposition \ref{P: polyhedral case}, (a).  Recall from the proposition that these are $\varphi$-infinity harmonic in $V$ and $f_{c,n}(\cdot - x_{0}) \to \varphi(\cdot - x_{0})$ uniformly in $V$ as $n \to \infty$.  
\vs	
	By the claimed convergence and \eqref{E: contradiction time}, we can find an $N \in \mathbb{N}$ large so that 
		\begin{equation} \label{E: contradiction closer}
			\max \{ w(x) - a f_{c,N}(x - x_{0}) \, :  \, x \in V\} > \max \{ w(y) - a f_{c,N}(y - x_{0}) \, :  \, y \in \partial V \}.
		\end{equation}
	At the same time, if we define $g_{N,\epsilon}$ by 
		\begin{equation*}
			g_{N,\epsilon}(x) = f_{c,N}(x - x_{0}) - \epsilon f_{c,N}(x- x_{0})^{2},
		\end{equation*}
	then, as $\epsilon \to 0^{+}$ and  uniformly in $V$,  $g_{N,\epsilon} \to f_{c,N}$.
	 Furthermore, $g_{N,\epsilon}$ is strictly $\varphi$-superharmonic.  Indeed, in $V - x_{0}$, we compute
		\begin{align*}
			- \langle D^{2} g_{N,\epsilon} \cdot \partial \varphi^{*}(D g_{N,\epsilon}), \partial \varphi^{*}(Dg_{N,\epsilon}) \rangle &= - (1 - 2 \epsilon f_{c,N}) \langle D^{2} f_{c,N} \cdot \partial \varphi^{*}(Df_{c,N}), \partial \varphi^{*}(Df_{c,N}) \rangle \\
			&\quad + 2\epsilon \langle Df_{c,N}, \partial \varphi^{*}(Df_{c,N}) \rangle^{2} \\
			&= 2 \epsilon \varphi^{*}(Df_{c,N})^{2} \geq 2 \epsilon.
		\end{align*}
	
	Since $a g_{N,\epsilon}(\cdot - x_{0})$ is a smooth strict supersolution for $\epsilon > 0$ small enough, we must have
		\begin{equation*}
			\max \left\{ w(x) - a g_{N,\epsilon}(x - x_{0}) \, :  \, x \in V \right\} \leq \max\left\{ w(y) - a g_{N,\epsilon}(y - x_{0}) \, :  \, y \in \partial V \right\},
		\end{equation*}
	which in the $\epsilon \to 0^{+}$ limit, contradicts \eqref{E: contradiction closer}.  \end{proof}  
	
Notice that the proof of Proposition \ref{P: easy cones} also works if $\varphi \in C^{2}(\mathbb{R}^{d} \setminus \{0\})$.  In that case, since $\varphi$ is already smooth enough, we can rerun the proof using $f_{c,n} \equiv \varphi$ independent of $c$ and $n$.  In fact, Proposition \ref{P: easy cones} holds in the generality of norms satisfying \eqref{A: running assumptions}.  This will be proved in Section \ref{S: variational}.

\subsection*{Smooth infinity harmonic approximations}  Here, we prove the first part of Proposition \ref{P: polyhedral case}.  
\vs
Since $\varphi$ is polyhedral, there is $N \in \mathbb{N}$ and vectors $p_{1},\dots,p_{N} \in \mathbb{R}^{d} \setminus \{0\}$ such that
	\begin{equation*}
		\varphi(q) = \max \left\{ \langle q, p_{1} \rangle, \dots, \langle q, p_{N} \rangle \right\}.
	\end{equation*}
Using Proposition \ref{P: canonical choice}, we can  assume that $\{p_{1},\dots,p_{N}\}$ is the set of extreme points of the dual ball $\{\varphi^{*} \leq 1\}$.
\vs
For each $i \in \{1,2,\dots,N\}$, let
	\begin{equation*}
		M_{i} = \left\{q \in \mathbb{R}^{d} \, :  \, \varphi(q) = \langle q, p_{i} \rangle \right\}.
	\end{equation*}
Notice that $\mathbb{R}^{d} = M_{1} \cup \dots \cup M_{N}$, and the relative interiors $\{\text{ri}(M_{1}),\dots,\text{ri}(M_{N})\}$ are non-empty and pairwise disjoint.
\vs
We will utilize the following fact.

	\begin{lemma} \label{L: discrete geometry}  For each $c > 0$, there is  $\epsilon_{c} > 0$ such that, if $q \in \{\varphi \geq c\}$ and $\mathcal{I}(q,\epsilon) = \{i \in \{1,2,\dots,N\} \, :  \, B(q,\epsilon) \cap M_{i} \neq \emptyset\}$, then
		\begin{equation*}
			\bigcap_{i \in \mathcal{I}(q,\epsilon)} M_{i} \neq \emptyset \ \  \text{if} \ \ \epsilon \in (0,\epsilon_{c}).
		\end{equation*}
	\end{lemma}  
	
\begin{proof} Although it is not hard to prove the claim ``by hand,"  we give here an alternate proof based on Durier's so-called Diff-Max property \cite{durier}, which is ``cleaner'' and  instructive since that property characterizes polyhedral Finsler norms.
\vs

A Finsler norm $\varphi : \mathbb{R}^{d} \to [0,\infty)$ is said to possess the Diff-Max property if, for each $q \in \mathbb{R}^{d}$, there is  $\epsilon_{q} > 0$ such that $\partial \varphi(q') \subseteq \partial \varphi(q)$ for all $q' \in B(q,\epsilon_{q})$.  An easy contradiction argument involving the upper semi-continuity of the subdifferential establishes that any polyhedral Finsler norm has the Diff-Max property --- see \cite[Theorem 5.1]{durier} for this and the converse.
\vs

Fix $c > 0$. Since  $\{\varphi = c\}$ is compact,  we can fix $\{q_{1},\dots,q_{N}\} \subseteq \{\varphi = c\}$ and $\{\epsilon_{1},\dots,\epsilon_{N}\} \subseteq (0,\infty)$, with $\epsilon_{i} = \epsilon_{q_{i}}$ the radius from the Diff-Max property, such that 
	\begin{equation*}
		\{\varphi = c\} \subseteq \bigcup_{i = 1}^{N} B(q_{i},\epsilon_{i}/2).
	\end{equation*}
Let $\epsilon' = \min\{\epsilon_{1},\dots,\epsilon_{N}\}$.  If $q \in \{\varphi = c\}$, then there is $i\in\{1,\ldots,d\}$ so that $B(q,\epsilon'/2) \subseteq B(q_{i},\epsilon_{i})$.  Hence, if $\epsilon \in (0,\epsilon'/2]$ and $j \in \mathcal{I}(q,\epsilon)$, then 
	\begin{equation*}
		\bigcap_{i \in \mathcal{I}(q,\epsilon)} M_{i} \supseteq \bigcap_{i \in \mathcal{I}(q_{i},\epsilon_{i})} M_{i} \neq \emptyset.
	\end{equation*}

Alternatively, if $\varphi(q) \geq c$, then 
	\begin{equation*}
		\bar{q} = \frac{c q}{\varphi(q)} \in \{\varphi = c\} \quad \text{and} \quad \mathcal{I}(q,\epsilon'/2) = \mathcal{I}(\bar{q}, (2\varphi(q))^{-1} (c\epsilon')).
	\end{equation*} 
Since $(2\varphi(q))^{-1} (c\epsilon') \leq \epsilon'/2$, the previous case implies that $\bigcap_{i \in \mathcal{I}(q,\epsilon')} M_{i} \neq \emptyset$. \end{proof}  
 
 We are now prepared to prove the first part of Proposition \ref{P: polyhedral case}. 

\begin{proof}[Proof of Proposition \ref{P: polyhedral case}, part (a)]  Given $\epsilon > 0$, define $f_{\epsilon} = \varphi * \eta_{\epsilon}$, where $(\eta_{\epsilon})_{\epsilon > 0}$ is a standard mollifying family, that is, $\eta_{\epsilon}(x) = \epsilon^{-d} \eta(\epsilon^{-1}x)$ with $\eta \in C^{\infty}_{c}(\mathbb{R}^{d})$  a nonnegative, even function such that $\int_{\mathbb{R}^{d}} \eta(x) \, dx = 1$.
\vs
In what follows, we fix $c > 0$ and $\epsilon \in (0,\epsilon_{c})$ with $\epsilon_{c}$ the constant from Lemma \ref{L: discrete geometry}. 

Notice that $D^{2} \varphi$ is the Radon measure given by 
	\begin{equation*}
		D^{2}\varphi(dq) = \frac{1}{2} \sum_{i = 1}^{N} \sum_{j \neq i} \|p_{i} - p_{j}\|^{-1} (p_{i} - p_{j}) \otimes (p_{i} - p_{j}) \mathcal{H}^{d-1} \restriction_{M(i) \cap M(j)}(dq).
	\end{equation*}
From this, we deduce that
	\begin{equation} \label{E: second derivative}
		D^{2} f_{\epsilon}(q) = \frac{1}{2} \sum_{i = 1}^{N} \sum_{j \neq i} \int_{M_{i} \cap M_{j}} \eta_{\epsilon}(q - q') \, \mathcal{H}^{d-1}(dq') \cdot \|p_{i} - p_{j}\|^{-1} (p_{i} - p_{j}) \otimes (p_{i} - p_{j}).
	\end{equation}
Similarly, we can write
	\begin{equation*}
		D\varphi(q) = \sum_{i = 1}^{N} p_{i} \chi_{M_{i}}(q),
	\end{equation*}
and, thus,
	\begin{equation*}
		Df_{\epsilon}(q) = \sum_{i = 1}^{N} \int_{M_{i}} \eta_{\epsilon}(q - q') \, dq' \cdot p_{i} 
	\end{equation*}
	
Fix $q \in \{\varphi \geq c\}$ and let $\tilde{\mathcal{I}} = \{i \in \{1,2,\dots,N\} \, :  \, \int_{M_{i}} \eta_{\epsilon}(q - q') \, dq' > 0\}$.  Notice that, in the notation of Lemma \ref{L: discrete geometry}, we have $\tilde{\mathcal{I}} \subseteq \mathcal{I}(q,\epsilon)$.  From this and Lemma \ref{L: discrete geometry}, it follows that 
	\begin{equation*}  
		\bigcap_{i \in \tilde{\mathcal{I}}} M_{i} \neq \emptyset.
	\end{equation*}
\vs
Up to renumbering, we can assume that $\{1,2,\dots,M\} = \tilde{\mathcal{I}}$ for some $M \leq N$.  The one-homogeneity of $\varphi$ implies that there is a $\bar{q} \in \{\varphi = 1\}$ such that	
	\begin{equation} \label{E: face found}
		\{p_{1},\dots,p_{M}\} \subseteq \{p \in \{\varphi^{*} \leq 1\} \, :  \, \langle p, \bar{q} \rangle = 1\},
	\end{equation}
and, hence, $\text{conv}(\{p_{1},\dots,p_{M}\}) \subseteq \{\varphi^{*} = 1\}$.  
\vs
Let $F$ be the smallest face of $\{\varphi^{*} \leq 1\}$ containing $\text{conv}(\{p_{1},\dots,p_{M}\})$ and observe that, in view of \eqref{E: face found}, we have  $F \neq \{\varphi^{*} \leq 1\}$.  Further, the relative interior of $\{p_{1},\dots,p_{M}\}$ is contained in $\text{ri}(F)$.  Otherwise, we could find a smaller face $F' \subsetneq \text{bdry} \, F$ such that $F' \supseteq \text{conv}(\{p_{1},\dots,p_{M}\})$. In particular, $Df_{\epsilon}(q) \in \text{ri}(F)$ and $\varphi^{*}(Df_{\epsilon}(q)) = 1$.
\vs
Finally, notice that, for each $i, j \in \tilde{\mathcal{I}}$, $p_{j} - p_{i} \in \mathcal{T}(Df_{\epsilon}(q))$.  Indeed, since $Df_{\epsilon}(q)$ is a relative interior point, we know that $F \subseteq p' + \mathcal{T}(Df_{\epsilon}(q))$ for each $p' \in F$.  Accordingly, $p_{j} - p_{i} \in \mathcal{T}(Df_{\epsilon}(q))$ holds.  
\vs
Observe that $\int_{M_{i} \cap M_{j}} \eta_{\epsilon}(q - q') \, \mathcal{H}^{d-1}(dq') > 0$ only if $\int_{M_{i}} \eta_{\epsilon}(q - q') \, dq' > 0$ and $\int_{M_{j}} \eta_{\epsilon}(q - q') \, dq' > 0$.  Thus, \eqref{E: second derivative} can be simplified to
	\begin{equation*}
		D^{2}f_{\epsilon}(q) = \frac{1}{2} \sum_{i \in \tilde{\mathcal{I}}} \sum_{j \in \tilde{\mathcal{I}} \setminus \{i\}}   \int_{M_{i} \cap M_{j}} \eta_{\epsilon}(q - q') \mathcal{H}^{d-1}(d q') \cdot \|p_{i} - p_{j}\|^{-1} (p_{i} - p_{j}) \otimes (p_{i} - p_{j}).
	\end{equation*}  
From this and the inclusion $p_{i} - p_{j} \in \mathcal{T}(Df_{\epsilon}(q))$ for $i,j \in \tilde{\mathcal{I}}$, we obtain, using  linearity,
	\begin{equation*}
		D^{2}f_{\epsilon}(q) \in \mathcal{S}_{\mathcal{T}(Df_{\epsilon}(q))}.
	\end{equation*}
In particular, $\partial \varphi^{*}(Df_{\epsilon}(q)) \subseteq \text{Ker} \, D^{2} f_{\epsilon}(q)$ by the definition of $\mathcal{T}(Df_{\epsilon}(q))$.
\vs
We conclude by setting $f_{c,n} = f_{\epsilon_{c} n^{-1}}$.  
\end{proof}  

\subsection*{Approximating norms}  We now use the convex functions of the previous subsection to build the norms of part (b) of Proposition \ref{P: polyhedral case}.  The idea is simply that the second derivative of a Finsler norm is determined by the second fundamental form of its unit ball.  Therefore, we only need to find convex sets that are appropriately curved, and the inclusion $D^{2}f_{c,n}(q) \in \mathcal{S}_{T(Df_{c,n}(q))}$ provides exactly this.

\begin{proof}[Proof of Proposition \ref{P: polyhedral case}, part (b)]  Fix $c \in (0,1)$ and apply part (a) to find $(f_{n})_{n \in \mathbb{N}}$ smooth convex functions such that, as $n\to \infty$,  $f_{n} \to f$ locally uniformly in $\{\varphi \geq c\}$ and \eqref{E: infinity harmonic} holds.  
\vs
Let $E_{n} = \{f_{n} \leq 1\}$, which is a convex set containing, for $n$ large enough,  the origin.  We claim that $\partial E_{n} \to \{\varphi = 1\}$ in the Hausdorff distance as $n \to \infty$.  Indeed, if $x_{n} \in \partial E_{n}$ for each $n \in \mathbb{N}$, then $f_{n}(x_{n}) \equiv 1$ and the local uniform convergence implies that any accumulation point of $(x_{n})_{n \in \mathbb{N}}$ is in $\{\varphi = 1\}$.  At the same time, if $x_{n} \in \{\varphi = 1\}$ for all $n$, then the identity $\varphi^{*}(Df_{c,n}) \equiv 1$ in $E_{n} \cap \{\varphi \geq c\}$ readily implies that $\text{dist}(x_{n},\partial E_{n}) \to 0$ as $n \to \infty$.  Hence the claim is proved.  
\vs
Let $(\varphi_{n})_{n \in \mathbb{N}}$ be the Finsler norms so that $\{\varphi_{n} \leq 1\} = E_{n}$ for each $n \in \mathbb{N}$.  Since $\partial E_{n}$ is smooth for each $n$, we have $(\varphi_{n})_{n \in \mathbb{N}} \subseteq C^{2}(\mathbb{R}^{d} \setminus \{0\})$.  We claim that, if $q \in \mathbb{R}^{d} \setminus \{0\}$ and $n \in \mathbb{N}$,  then $D^{2}\varphi_{n}(q) \subseteq \mathcal{S}(D\varphi_{n}(q),\varphi)$.
\vs
To see this, fix $n \in \mathbb{N}$ and $q \in \{\varphi_{n} = 1\} = \partial E_{n}$.  In view of the fact that  $\{\varphi_{n} = 1\} = \{f_{n} = 1\}$, the second fundamental forms coincide, that is,
	\begin{align*}
		D^{2}\varphi_{n}(q) &= \left(\text{Id} - \widehat{D\varphi_{n}}(q) \otimes \widehat{D\varphi_{n}}(q) \right) D^{2}\varphi_{n}(q) \left(\text{Id} - \widehat{D\varphi_{n}}(q) \otimes \widehat{D\varphi_{n}}(q) \right) \\
				&= \left(\text{Id} - \widehat{Df_{n}}(q) \otimes \widehat{Df_{n}}(q) \right) D^{2}f_{n}(q) \left(\text{Id} - \widehat{Df_{n}}(q) \otimes \widehat{Df_{n}}(q) \right).
	\end{align*}
Since  $D^{2}f_{n}(q) \in \mathcal{S}_{\mathcal{T}(Df_{n}(q))}$, we  find $D^{2}\varphi_{n}(q) \in \mathcal{S}(Df_{n}(q),\varphi)$.  Indeed, if $v \in \mathcal{T}(Df_{n}(q))$, then $v - \langle v, \widehat{Df_{n}}(q) \rangle \widehat{Df_{n}}(q) \in \langle Df_{n}(q) \rangle \oplus \mathcal{T}(Df_{n}(q))$ and, thus,
	\begin{equation*}
		\left(\text{Id} - \widehat{Df_{n}}(q) \otimes \widehat{Df_{n}}(q) \right) (v \otimes v) \left(\text{Id} - \widehat{Df_{n}}(q) \otimes \widehat{Df_{n}}(q) \right) \in \mathcal{S}(Df_{n}(q),\varphi).
	\end{equation*}
The fact that  the desired inclusion property holds for elementary tensors of the form $v \otimes v \in \mathcal{S}_{\mathcal{T}(Df_{n}(q))}$ and linearity yield that it holds for all operators in $\mathcal{S}_{\mathcal{T}(Df_{n}(q))}$.
\end{proof}  

\section{Discrete $p$-Laplace-like Schemes in General Lattices} \label{S: p laplace}

In this section, we show that the PDEs from \cite{chatterjee_souganidis}, which include the $p$-Laplace-like family of PDE \eqref{E: l1 median}, are compatible with Finsler norms and, hence,  admit comparison principles.   As we shall see, the results of \cite{chatterjee_souganidis} remain true if we replace $\mathbb{Z}^{d}$ by some other lattice in $\mathbb{R}^{d}$.  This leads to a whole class of operators analogous to \eqref{E: p laplace} that, although discontinuous in general, are compatible with Finsler geometries derived from the choice of lattice.
\vs
For completeness, we first recall the basic set up and schemes studied in \cite{chatterjee_souganidis}.

\subsection*{The discrete schemes in \cite{chatterjee_souganidis}.}  To start with, we fix a lattice $\Lambda \subseteq \mathbb{R}^{d}$, that is, a rank $d$ subgroup of $\mathbb{R}^{d}$,  and we choose a finite symmetric subset $E \subseteq \Lambda \setminus \{0\}$;
symmetry here means that $-E = E$.  We think of $(\Lambda,E)$ as defining a graph so that $(x,y)$ is an edge if $x - y \in E$.  
\vs
Associated to this graph there is a discrete gradient $D^{E}$ for functions $v : \Lambda \to \mathbb{R}$ given by 
	\begin{equation*}
		D^{E}v(x) = \{v(x + e) - v(x)\}_{e \in E}.
	\end{equation*}

Finally, for each $\alpha \in [1,\infty]$, we define operators $M_{\alpha}^{E} : \mathbb{R}^{E} \to \mathbb{R}$ by 
	\be\label{E: max min}
	\begin{split}
		M_{1}^{E}(V) &= \text{Med}\{v(x + e) - v(x) \, :  \, e \in E\},\\[1.2mm]
		M_{\alpha}^{E}(V) &= \text{argmin}_{y \in \mathbb{R}} \sum_{e \in E} |V(e) - y|^{\alpha} \quad \text{if} \, \, \alpha < \infty, \\
		M_{\infty}^{E}(V) &= \lim_{\alpha \to \infty} M_{\alpha}^{E}(V) = \frac{1}{2} \max_{e \in E} V(e) + \frac{1}{2} \min_{e \in E} V(e). 
	\end{split}
	\ee
Recall that the median $\text{Med}(A)$ of a finite set $A$ is defined in Section \ref{S: median}.  When $\alpha \in (1,\infty)$, the minimum in the definition of $M_{\alpha}$ is uniquely attained due to strict convexity.  Thus, the operators $\{M^{E}_{\alpha}\}_{\alpha \in (1,\infty)}$ are well-defined.
\vs
	
We are interested in the limiting PDE obtained from schemes in $(\Lambda,E)$ defined in the following way: given $u^{(0)} : \Lambda \to \mathbb{R}$,  the sequence $(u^{(n)})_{n \in \mathbb{N} \cup \{0\}}$ is defined recursively by the rule
	\begin{equation} \label{E: scheme}
		u^{(n)}(x) = u^{(n - 1)}(x) + M_{\alpha}^{E}(\nabla^{E}u^{(n-1)}(x)) \quad \text{if} \, \, x \in \Lambda.
	\end{equation}

\subsection*{Limiting PDE}  It was shown in \cite{chatterjee_souganidis} that, under suitable assumptions on $(\Lambda,E)$, the parabolic scaling of $(u^{(n)})_{n \in \mathbb{N} \cup \{0\}}$ yields, in the limit $n\to \infty$, a whole class of PDE generalizing \eqref{E: p laplace}.  It turns out that, in every case, the piecewise linear geometry of the underlying lattice manifests itself in the structure of the PDE. The convergence was conditioned upon knowing that the limit problem admits a comparison principle, a fact which follows from this paper.
\vs

Before stating the result, it is convenient to define the mathematical objects appearing in the limit.  First, let $J, L : \mathbb{R}^{d} \to \mathcal{P}(E)$ be the set-valued maps 
	\begin{equation} \label{E: weird indices}
		J(p) = \text{argmax}_{e \in E} \langle p, e \rangle \ \ \text{and} \ \ L(p) = \text{argmin}_{e \in E} \langle p,e \rangle.
	\end{equation}
Next, we define the maps $\overline{F}_{E,\alpha}, \underline{F}^{E,\alpha} : \mathbb{R}^{d} \times \mathcal{S}^{d} \to \mathbb{R}$.
\vs
When $\alpha = 1$, $\overline{F}_{E,1}$ and $\underline{F}^{E,1}$ are given by
\[		\overline{F}_{E,1}(p,X) = \max \left\{ \langle X e, e \rangle \, :  \, e \in L(p) \right\} \ \ \text{and} \ \ 
		\underline{F}^{E,1}(p,X) = \min \left\{ \langle X e, e \rangle \, :  \, e \in L(p) \right\}.\]
An analogous formula holds when $\alpha = \infty$, namely,
	\be\label{E: alpha infinity}
	\begin{split}
		\overline{F}_{E,\infty}(p,X) &= \max \left\{ \langle X e, e \rangle \, :  \, e \in J(p) \right\},\\[1.2mm]
		\underline{F}^{E,\infty}(p,X) &= \min \left\{ \langle X e, e \rangle \, :  \, e \in J(p) \right\}.
	\end{split}
	\ee

As discussed in Remark \ref{R: ambiguous}, the pair $(\overline{F}_{E,\infty},\underline{F}_{E,\infty})$ encodes a Finsler infinity Laplacian.  This is not true of $(\overline{F}_{E,1},\underline{F}_{E,1})$ as shown in Section \ref{S: not finsler} in the appendices. 
\vs

Finally, when $\alpha \in (1,\infty)$, it is convenient first to define $G \subseteq \mathbb{R}^{d}$ by
	\begin{equation*}
		G = \{ p \in \mathbb{R}^{d} \, :  \, \min_{e \in E} |\langle p, e \rangle| > 0\}. 
	\end{equation*}
Then $F_{E,\alpha} : G \times \mathcal{S}^{d} \to \mathbb{R}$ is given by 
	\begin{equation*}
		F_{E,\alpha}(p,X) = \Big(\sum_{e \in E} |\langle p,e \rangle|^{\alpha - 2}\Big)^{-1} \sum_{e \in E} |\langle p, e \rangle|^{\alpha - 2} \langle X e, e \rangle.
	\end{equation*}
The operators $\overline{F}_{E,\alpha}$ and $\underline {F}^{E,\alpha}$ are then the semi-continuous envelopes of $F_{E,\alpha}$, that is,
	\begin{align*}
		\overline{F}_{E,\alpha}(p,X) &= \lim_{\delta \to 0^{+}} \sup \left\{ F_{E,\alpha}(p',X) \, :  \, p' \in G, \, \,  \|p' - p\| \leq \delta \right\}, \\
		\underline{F}^{E,\alpha}(p,X) &= \lim_{\delta \to 0^{+}} \inf \left\{ F_{E,\alpha}(p',X) \, :  \, p' \in G, \, \, \|p' - p\| \leq \delta \right\}.
	\end{align*}
Notice that, when $\alpha \geq 2$, $\overline{F}_{E,\alpha} \equiv \underline{F}^{E,\alpha}$ in $(\mathbb{R}^{d} \setminus \{0\}) \times \mathcal{S}^{d}$.  This is false when $\alpha < 2$.  
\vs

Concerning the parabolic scaling limit of \eqref{E: scheme}, here is the main result, the second part of which is due to \cite{chatterjee_souganidis}. 

	\begin{theorem}  Suppose that $\Lambda$ is a rank $d$ subgroup of $\mathbb{R}^{d}$ and $E \subseteq \Lambda \setminus \{0\}$ is a finite subset such that $-E = E$.  Assume, in addition, that $E$ generates $\Lambda$, that is,
		\begin{equation} \label{E: generation}
			\Lambda = \Big\{\sum_{i = 1}^{N} m_{i} e_{i} \, :  \, N \in \mathbb{N}, \, \, m_{1},\dots,m_{N} \in \mathbb{Z}, \, \, e_{1},\dots,e_{N} \in E \Big\}.
		\end{equation}
	Given $u_{0} \in BUC(\mathbb{R}^{d})$, if the sequences $(u^{(n)}_{\epsilon})_{n \in \mathbb{N} \cup \{0\}}$ are defined so that $u^{(0)}_{\epsilon}(x) = u_{0}(\epsilon x)$ and $u^{(n)}_{\epsilon}$ satisfies \eqref{E: scheme} for each $\epsilon > 0$, then
		\begin{itemize}
			\item[(i)] there is a unique viscosity solution $u \in BUC(\mathbb{R}^{d} \times [0,T])$ of \eqref{E: parabolic} with $\overline{F} = \overline{F}_{E,\alpha}$ and $\underline{F} = \underline{F}^{E,\alpha}$.
			\item[(ii)] $u^{(n)}_{\epsilon} \to u$, that is, 
				\begin{equation*}
					\lim_{\delta \to 0^{+}} \sup \left\{ |u^{(n)}_{\epsilon}(x) - u(y,s)| \, :  \, (y,s) \in \mathbb{R}^{d} \times [0,T], \, \, \|\epsilon x - y\| + |\epsilon^{2} n - s| + \epsilon \leq \delta \right\} = 0.
				\end{equation*}
		\end{itemize}  
	\end{theorem}  

The remainder of the Section is devoted to the proof of assertion (i) of the theorem.  
\vs
When $\alpha \in \{1\} \cup [2,\infty]$, assertion (ii) follows as in \cite{chatterjee_souganidis}.  The range $\alpha \in (1,2)$ is admittedly a more challenging computation.  Nonetheless, since the focus of this paper is comparison results for the limiting PDE, we leave the remaining details as an exercise for the interested reader.
\vs
Notice that, when $\Lambda = \mathbb{Z}^{d}$ and $E = \{\rho \bar{e}_{i} \, :  \, i \in \{1,2,\dots,d\}, \, \, \rho \in \{-1,1\}\}$, the limiting PDE are precisely \eqref{E: l1 median} and \eqref{E: l1 infinity caloric}.

\subsection*{The case $\alpha = \infty$}  As is already suggested by \eqref{E: max min}, when $\alpha = \infty$, the limiting operators $\overline{F}_{E,\infty}$ and $\underline{F}^{E,\infty}$ encode an infinity Laplacian.  
\vs
To see this, it suffices to prove the following lemma:

	\begin{lemma} \label{L: alpha infinity} Let $\varphi_{E}$ be the polyhedral norm with  dual $\varphi_{E}^{*}(p) = \max \left\{ \langle p, e \rangle \, :  \, e \in E \right\}$, the latter being a norm by \eqref{E: generation}.  Let $(\overline{G}_{\varphi_{E}},\underline{G}^{\varphi_{E}})$ be the associated infinity Laplacian operators given by \eqref{E: infinity laplace}.  
	If $p \in \mathbb{R}^{d}$ and $X \in \mathcal{S}(p,\varphi_{E})$, then
		\begin{equation} \label{E: ambiguous}
			\underline{G}^{\varphi_{E}}(p,X) \leq \overline{F}_{E,\infty}(p,X) \leq \underline{F}^{E,\infty}(p,X) \leq \overline{G}_{\varphi_{E}}(p,X).
		\end{equation}
	\end{lemma}  

The lemma implies, in particular, that $(\overline{F}_{E,\infty},\underline{F}^{E,\infty})$ satisfy \eqref{A: consistency} with the norm $\varphi_{E}$.  Note that the other assumptions \eqref{A: elliptic_1}, \eqref{A: elliptic_2}, and \eqref{A: semicontinuous} hold trivially.  Therefore, the comparison principle applies by Theorem \ref{T: main parabolic} and Theorem \ref{T: polyhedral norms}.

	\begin{proof}[Proof of Lemma \ref{L: alpha infinity}]  This is a direct computation.  The main point is that
		\begin{equation*}
			\partial \varphi_{E}^{*}(p) = \left\{q \in \text{conv}(E) \, :  \, \langle q,p \rangle = \varphi^{*}(p) \right\}.
		\end{equation*} \end{proof}  
	
\begin{remark} \label{R: ambiguous} The previous lemma exposes a fundamental reality in the approach taken thoughout this paper. The operators $\overline{F}$ and $\underline{F}$ associated with \eqref{E: parabolic} are usually ambiguously defined.  For instance, in the present example, the solutions of \eqref{E: parabolic} with the operators $(\overline{F}_{E,\infty},\underline{F}^{E,\infty})$ are the same as those of \eqref{E: parabolic} with $(\overline{G}_{\varphi_{E}},\underline{G}^{\varphi_{E}})$, even though the inequalities in \eqref{E: ambiguous} can be strict for some vectors $(p,X)$.
\vs
Note, however, that $\overline{F}$ and $\underline{F}$ are unambiguously defined whenever $X \in \mathcal{S}(p,X)$.  (Indeed, they even coincide by fiat due to \eqref{A: consistency}.)  The proof of Theorem \ref{T: main parabolic} shows that the only information that is necessary to derive the comparison principle is the values of these functions on pairs $(p,X)$ with $X \in \mathcal{S}(p,\varphi)$.  We could have restricted the domains of $\overline{F}$ and $\underline{F}$ to the set of such vectors from the outset, and even restricted our definitions of viscosity sub- and super-solutions accordingly.  For an exposition following that approach in the context of level set PDEs,  see Barles and Georgelin \cite{barles_georgelin}.  \end{remark}

\subsection*{The case $\alpha \in [2,\infty)$}  As it was pointed out already, when $\alpha \in [2,\infty)$, the comparison principle for the associated equation follows classically.  
\vs
To put it another way, the operators $(\overline{F}_{E,\alpha},\underline{F}_{E,\alpha})$ satisfy \eqref{E: mean curvature type} in this regime so, by Proposition \ref{P: strictly convex norms}, they are compatible with the Euclidean norm.  Hence comparison is implied by Theorem~\ref{T: main parabolic} and Theorem~\ref{T: C2 norms}.

\subsection*{The case $\alpha \in [1,2)$}  In this setting, the geometry is more complicated.  We define the norm $\underline{\varphi}_{E}$ implicitly through its dual norm which is given by 
	\begin{equation} \label{E: derived norm}
		\underline{\varphi}^{*}_{E}(p) = \max \Big\{ \sum_{e' \in E \setminus \{e,-e\}} |\langle p, e' \rangle| \, :  \, e \in E \Big\}.
	\end{equation}
We leave it to the reader to check that $\underline{\varphi}^{*}_{E}$ defines a polyhedral Finsler norm.  Recall that $\underline{\varphi}_{E}$ is polyhedral by Proposition \ref{P: polyhedral norms duality}.
\vs
As in the case $\alpha = \infty$, to apply Theorem~\ref{T: main parabolic} and Theorem~\ref{T: polyhedral norms}, all that is left is to prove that $(\overline{F}_{E,\alpha},\underline{F}^{E,\alpha})$ is compatible with the norm $\underline{\varphi}_{E}$.  This is implied by the following result.

	\begin{prop} \label{P: alpha less infinity} If $ \alpha \in [1,2)$, then,  for each $p \in \mathbb{R}^{d}$, if $X \in \mathcal{S}(p,\underline{\varphi}_{E})$,
		\begin{equation*}
			\overline{F}_{E,\alpha}(p,X) = \underline{F}^{E,\alpha}(p,X).		\end{equation*}
	\end{prop}  
	
Before proceeding to the proof, we remark that we do not know an enlightening justification for the appearance of the norm \eqref{E: derived norm} or a geometric interpretation of the map $(\Lambda,E) \mapsto \varphi_{E} \mapsto \varphi_{E}^{*}$. Indeed, we only know that it works in the computations that follow.  Nonetheless, it may be worth explaining where it comes from.
\vs
When $\Lambda = \mathbb{Z}^{d}$ and $E$ consists of the standard orthonormal basis together with its antipodal points, the operators $(\overline{F}_{E,1},\underline{F}_{E,1})$ are determined by 
	\begin{equation*}
		\overline{F}_{E,1}(p,X) = \underline{F}_{E,1}(p,X) = X_{ii} \ \ \text{if} \ \ |p_{i}| < \min\{|p_{1}|,\dots,|p_{i-1}|,|p_{i+1}|,\dots,|p_{d}|\}.
	\end{equation*}
Consider the $d = 3$ setting.  If we visualize the discontinuity set in the gradient variable, ignoring the Hessian, this brings to mind the cube $\{p \in \mathbb{R}^{d} \, :  \, |p_{1}|,\dots,|p_{d}| \leq 1\}$ with each $2$-dimensional face subdivided into four congruent triangles.  We can associate to each triangle an opposing one that shares a boundary along one of the $1$-dimensional faces of the cube.  The operators $(\overline{F}_{E,1},\underline{F}_{E,1})$ are then constant in the union of any triangle and its opposing one.
\vs
As for the relevance of \eqref{E: derived norm}, we notice that, up to a constant factor,  it coincides with the example \eqref{E: rhombic dodecahedron} of Section \ref{S: polyhedral norms}, and the unit ball $\{\underline{\varphi}^{*} \leq 1\}$ is a rhombic dodecahedron.  One of the classical constructions of the rhombic dodecahedron proceeds by starting with the cube and adding a pyramid to each face --- see \cite{mathworld_article} or \cite{steinhaus}.  Projecting the pyramids down onto the cube results in exactly the triangles mentioned in the previous paragraph.   

\subsection*{Proof of Proposition \ref{P: alpha less infinity}}  To start with, it will be convenient to define $\tilde{E} \subseteq \Lambda$ by
	\begin{equation*}
		\tilde{E} = \Big\{ \sum_{e' \in E \setminus \{e,-e\}} \rho(e') e' \, :  \, e \in E, \, \, \rho \in \{-1,1\}^{E \setminus \{e\}} \Big\}.
	\end{equation*}
Given $p \in \mathbb{R}^{d}$, we let $\tilde{E}(p)$ denote the subset of $\tilde{E}$ given by
	\begin{equation*}
		\tilde{E}(p) = \{q \in \tilde{E} \, :  \, \langle p, q \rangle = \underline{\varphi}_{E}^{*}(p) \}.
	\end{equation*}
\vs
The following result relates $\tilde{E}$ to $\underline{\varphi}_{E}^{*}$, $L$ (see \eqref{E: weird indices}), and $\mathcal{T}(\cdot,\underline{\varphi}_{E})$ and will be useful in the sequel:

	\begin{prop} \label{P: convexity facts} (i) For each $p \in \mathbb{R}^{d}$,   
	$\underline{\varphi}_{E}^{*}(p) = \max \{ \langle p, q \rangle \, :  \, q \in \tilde{E}\}$ and $\partial \underline{\varphi}_{E}^{*}(0) = \text{conv}(\tilde{E})$.  
	\vskip.05in
	(ii) Given $p \in \mathbb{R}^{d} \setminus \{0\}$,  $e \in L(p)$ if and only if $\underline{\varphi}_{E}^{*}(p) = \sum_{e' \in E \setminus \{e,-e\}} |\langle p, e' \rangle|.$
	Furthermore, if $e, e' \in L(p)$, then $|\langle p,e \rangle| = |\langle p,e' \rangle|$.
	\vskip.05in
	(iii) For each $p \in \mathbb{R}^{d}$,  $\mathcal{T}(p,\underline{\varphi}_{E}) = \tilde{E}(p)^{\perp}$.
\end{prop}  

	\begin{proof}  (i) The formula for $\underline{\varphi}^{*}_{E}$ follows directly from its definition and that of $\tilde{E}$.  The inclusion $\partial \underline{\varphi}_{E}^{*}(0) \supseteq \text{conv}(\tilde{E})$ follows immediately.  The opposite inclusion can be proved by contrapositive using separating hyperplanes.  
\vs	
	(ii) If $e \in L(p)$, then, for each $e' \in E$,
			\begin{align*}
				\sum_{e'' \in E \setminus \{e,-e\}} |\langle p,e'' \rangle| &= 2|\langle p,e' \rangle| + \sum_{e'' \in E \setminus \{e,e',-e,-e'\}} |\langle p,e'' \rangle| \\
						&\geq 2|\langle p, e \rangle| + \sum_{e'' \in E \setminus \{e,e',-e,-e'\}} |\langle p,e'' \rangle|= \sum_{e'' \in E \setminus \{e',-e'\}} |\langle p, e'' \rangle|
			\end{align*}
		with equality if and only if $|\langle p, e' \rangle| = |\langle p,e \rangle|$.  Hence $\underline{\varphi}^{*}_{E}(p) = \sum_{e'' \in E \setminus \{e,-e\}} |\langle p,e'' \rangle|$ and the map $v \mapsto |\langle p, v \rangle|$ is constant in $L(p)$.
\vs 		
	(iii) First, we claim that $\tilde{E}(p) \subseteq \partial \underline{\varphi}^{*}_{E}(p)$.  Indeed, since $\tilde{E} \subseteq \partial \underline{\varphi}_{E}^{*}(0)$, we know that
				\begin{equation*}
					\partial \underline{\varphi}_{E}^{*}(p) = \{q \in \partial \underline{\varphi}_{E}^{*}(0) \, :  \, \underline{\varphi}_{E}^{*}(p) = \langle p, q \rangle \} \supseteq \tilde{E}(p).
				\end{equation*}
			Thus, $\mathcal{T}(p,\underline{\varphi}_{E}) = \partial \underline{\varphi}_{E}^{*}(p)^{\perp} \subseteq \tilde{E}(p)^{\perp}$.  
\vs			
			To prove the opposite inclusion, we first show that $\partial \underline{\varphi}_{E}^{*}(p) = \text{conv}(\tilde{E}(p))$.  What we already proved implies that $\text{conv}(\tilde{E}(p)) \subseteq \partial \underline{\varphi}_{E}^{*}(p)$.  To see that equality holds, observe that, if $q \in \partial \underline{\varphi}_{E}^{*}(p)$, then $q \in \partial \underline{\varphi}_{E}^{*}(0) = \text{conv}(\tilde{E})$. It follows that we can fix $\{q_{1},\dots,q_{N}\} \subseteq \tilde{E}$ and $\{\lambda_{1},\dots,\lambda_{N}\} \subseteq [0,1]$ such that $\sum_{i =1}^{N} \lambda_{i} = 1$ and $q = \sum_{i = 1}^{N} \lambda_{i} q_{i}$.  Moreover, $q \in \partial \underline{\varphi}_{E}^{*}(p)$ yields 
				\begin{equation*}
					\underline{\varphi}_{E}^{*}(p) = \langle q, p \rangle = \sum_{i = 1}^{N} \lambda_{i} \langle q_{i}, p \rangle \leq \sum_{i = 1}^{N} \lambda_{i} \underline{\varphi}_{E}^{*}(p) = \underline{\varphi}_{E}^{*}(p).
				\end{equation*}
			It follows that $\langle q_{i},p \rangle = \underline{\varphi}_{E}^{*}(p)$ for each $i \in \{1,2,\dots,N\}$.  In other words, $q \in \text{conv}(\tilde{E}(p))$, and, hence,  $\partial \underline{\varphi}_{E}^{*}(p) \subseteq \text{conv}(\tilde{E}(p))$ 
\vs			
			Finally, suppose that $v \in \tilde{E}(p)^{\perp}$.  By definition, this means the linear functional\\ $\ell_{v} : \mathbb{R}^{d} \to \mathbb{R}$ given by $\ell_{v}(q) = \langle q, v \rangle$ vanishes in $\tilde{E}(p)$.  Hence it vanishes in $\text{conv}(\tilde{E}(p)) = \partial \underline{\varphi}_{E}^{*}(p)$, which is to say that $v \in \partial \underline{\varphi}_{E}^{*}(p)^{\perp}$.  Thus, $\tilde{E}(p)^{\perp} \subseteq \partial \underline{\varphi}_{E}^{*}(p)^{\perp} = \mathcal{T}(p, \underline{\varphi}_{E})$.\end{proof}  

We are now prepared to state and prove the main identity underlying Proposition \ref{P: alpha less infinity}.  In what follows, given $p \in \mathbb{R}^{d} \setminus \{0\}$, define $\rho_{p} : E \to \{-1,1\}$ and $q_{p} : E \to \tilde{E}$ by 
				\begin{align*}
					q_{p}(e) = \sum_{e' \in E \setminus \{e,-e\}} \rho_{p}(e') e' \ \ \text{and} \ \ \rho_{p}(e) = \left\{ \begin{array}{r l}
											\dfrac{\langle p,e \rangle}{|\langle p,e \rangle|} & \text{if} \, \, |\langle p,e \rangle| > 0, \\
											1 & \text{otherwise.}
									\end{array} \right.
				\end{align*}
			
		\begin{prop} \label{P: transition rates} Let $p \in \mathbb{R}^{d} \setminus \{0\}$ and suppose that $v \in \mathcal{T}(p,\underline{\varphi}_{E})$.  For each $e, e' \in L(p)$, we have
			\begin{equation*}
				\rho_{p}(e) \langle p, e \rangle = \rho_{p}(e') \langle p,e' \rangle \ \ \text{and} \ \  \rho_{p}(e) \langle v, e \rangle = \rho_{p}(e') \langle v, e' \rangle.
			\end{equation*}		
		\end{prop}  
		
			\begin{proof}  First, we prove the identity for $v$.  Given $e, e' \in L(p)$, Proposition \ref{P: convexity facts} implies that $q_{p}(e), q_{p}(e') \in \tilde{E}(p)$.  Thus, by the definition of $\mathcal{T}(p, \underline{\varphi}_{E})$,
		\begin{equation*}
			0 = \langle v, q_{p}(e) \rangle = \sum_{e'' \in E \setminus \{e,-e\}} \rho_{p}(e'') \langle v, e'' \rangle,
		\end{equation*}
	and 
		\begin{equation*}
			2\rho_{p}(e) \langle v, e \rangle = - \sum_{e'' \in E \setminus \{e,e',-e,-e'\}} \rho_{p}(e'') \langle v, e'' \rangle = 2\rho_{p}(e') \langle v, e' \rangle.
		\end{equation*}
		
	Next, we invoke Proposition \ref{P: convexity facts} to see that the inclusion $\{e, e'\} \subseteq L(p)$ implies  $|\langle p, e \rangle| = |\langle p, e' \rangle|$.  Hence, by the definition of $\rho_{p}$,
		\begin{equation*}
			\rho_{p}(e) \langle p, e \rangle = |\langle p,e \rangle| = |\langle p,e' \rangle| = \rho_{p}(e') \langle p, e' \rangle.
		\end{equation*} 
		\end{proof}
	
All that remains is to prove Proposition \ref{P: alpha less infinity}.

\begin{proof}[Proof of Proposition \ref{P: alpha less infinity}]  To start with, we claim that, if $X \in \mathcal{S}(p,\underline{\varphi}_{E})$, then the quadratic form $v \mapsto \langle Xv, v \rangle$ is constant in $L(p)$.  Note that, by linearity, it suffices to prove this when $X$ is an element of the spanning set
			\begin{equation*}
				\{p \otimes v + v \otimes p \, :  \, v \in \mathcal{T}(p,\underline{\varphi}_{E}) \} \cup \{v \otimes v' + v' \otimes v \, :  \, v, v' \in \mathcal{T}(p,\underline{\varphi}_{E})\} \cup \{p \otimes p\}.
			\end{equation*}
		For such tensors, the desired identity follows directly from Proposition \ref{P: transition rates}. 
	Indeed, if $u, u' \in \{p\} \cup \mathcal{T}(p,\underline{\varphi}_{E})$ and $e, e' \in L(p)$, then
			\begin{equation*}
				\langle (u \otimes u' + u' \otimes u) e, e \rangle = 2\rho_{p}(e)^{2} \langle u, e \rangle \langle u', e \rangle = 2\rho_{p}(e')^{2} \langle u, e' \rangle \langle u', e' \rangle = \langle (u \otimes u' + u' \otimes u)e', e' \rangle.
			\end{equation*}
		This proves that the quadratic form induced by $u \otimes u' + u' \otimes u$ is constant in $L(p)$.  
		\vs
	From the previous computation, we see that, if $X \in \mathcal{S}(p,\underline{\varphi}_{E})$, then the identity $\overline{F}_{E,1}(p,X) = \underline{F}^{E,1}(p,X)$ follows immediately from the definitions.  If $\alpha \in (1,2)$, then discontinuities only occur when $p \notin G$, that is, when $v \mapsto |\langle v, p \rangle|$ vanishes in $L(p)$.  On the other hand, for such $\alpha$, a short calculation shows that, if $p \in \mathbb{R}^{d} \setminus G,$ then 
		\begin{equation*}
			\overline{F}_{E,\alpha}(p,X) = \overline{F}_{E,1}(p,X) \ \ \text{and} \ \ \underline{F}^{E,\alpha}(p,X) = \underline{F}^{E,1}(p,X).
		\end{equation*}
	Hence, once again, $\overline{F}_{E,\alpha}(p,X) = \underline{F}^{E,\alpha}(p,X)$ holds if $X \in \mathcal{S}(p,\underline{\varphi}_{E})$.  
\end{proof}  

\section{$L^{\infty}-$Variational Problems} \label{S: variational}

Here we discuss some of the consequences of our results  for the theory of $L^{\infty}-$variational problems involving the infinity Laplacian associated with a Finsler norm.
\vs
Throughout this section, we will avail ourselves of the results in Appendix \ref{Appendix: 2d}, especially Proposition~\ref{P: elliptic 2d norm} and Proposition~\ref{P: 2d case}.  These provide comparison results for elliptic PDE involving the infinity Laplacian operator associated with an arbitrary Finsler norm in $\mathbb{R}^{2}$.  

\subsection*{$\varphi$-infinity harmonic functions}  Recall that in Section \ref{S: gimmick cone comparison}, we showed that the cone comparison properties of the variational approach to $\varphi$-infinity harmonic functions are equivalent to the viscosity theoretic ones when $\varphi$ is polyhedral or $C^{2}$.  We now show, by a slightly more cumbersome but similar argument, that this is true whenever $\varphi$ is ``nice enough," that is, \eqref{A: running assumptions} holds.  

\begin{theorem} \label{T: cone comparison}  If $\varphi$ is a Finsler norm in $\mathbb{R}^{d}$ satisfying  \eqref{A: running assumptions}, then $w \in CCA_{\varphi}(\Omega)$ (resp.\ $v \in CCB_{\varphi}(\Omega)$) if and only if 
				\begin{gather*}
					- \langle D^{2} w \cdot \partial \varphi^{*}(Dw), \partial \varphi^{*}(Dw) \rangle \leq 0 \ \ \text{in} \ \ \Omega \\
					(\text{resp.} \, \, - \langle D^{2} v \cdot \partial \varphi^{*}(Dv), \partial \varphi^{*}(Dv) \rangle \geq 0 \ \ \text{in} \ \ \Omega).
				\end{gather*}
			\end{theorem}  
			
			Before proceeding to the proof of Theorem \ref{T: cone comparison}, we state a well-known fact from the theory of $L^{\infty}-$variational problems, which will be used in the sequel.  Here, we give a somewhat unconventional (but easy) viscosity theoretic proof.

\begin{lemma} \label{L: infinity harmonic cones}  If $\varphi$ is a Finsler norm in $\mathbb{R}^{d}$, then $- \langle D^{2} \varphi \cdot \partial \varphi^{*}(D\varphi), \partial \varphi^{*}(D\varphi) \rangle = 0$ in the viscosity sense in $\mathbb{R}^{d} \setminus \{0\}$.  \end{lemma}

It is possible to prove the lemma by first showing that $\varphi$ has both cone comparison properties and invoking the well-known ``if" implication of Theorem \ref{T: cone comparison}.  Note that this can also be checked ``by hand," invoking the definition of viscosity solution directly and employing elementary convex analytic arguments.

	\begin{proof}  Let $(\varphi_{n})_{n \in \mathbb{N}} \subseteq C^{2}(\mathbb{R}^{d} \setminus \{0\})$ be a sequence of Finsler norms such that, as $n \to \infty$, $\varphi_{n} \to \varphi$ locally uniformly in $\mathbb{R}^{d}$.  By Proposition \ref{P: smooth case}, these norms satisfy, for each $n \in \mathbb{N}$,
		\begin{equation*}
			- \langle D^{2}\varphi_{n} \cdot \partial \varphi^{*}_{n}(D\varphi_{n}), \partial \varphi^{*}_{n}(D\varphi_{n})  \rangle = 0 \ \ \text{in} \ \ \mathbb{R}^{d} \setminus \{0\}.
		\end{equation*}
	Notice that $\varphi_{n}^{*} \to \varphi^{*}$ locally uniformly in $\mathbb{R}^{d}$.  Therefore, using Proposition \ref{P: upper semi-continuity} and a standard stability argument, we conclude that $-\langle D^{2} \varphi \cdot \partial \varphi^{*}(D\varphi), \partial \varphi^{*}(D\varphi) \rangle = 0$ in $\mathbb{R}^{d} \setminus \{0\}$.  \end{proof}

\begin{proof}[Proof of Theorem \ref{T: cone comparison}]  We will prove the statement for $\varphi$-infinity subharmonic functions since the statement for superharmonic functions follows as a consequence.
\vs
If $w \in CCA_{\varphi}(\Omega)$, then,  arguing exactly as in \cite{aronsson_crandall_juutinen}, yields that $w$ satisfies the differential inequality
	\begin{equation} \label{E: infinity subharmonic}
		-\langle \partial \varphi^{*}(Dw), D^{2} w \partial \varphi^{*}(Dw) \rangle \leq 0.
	\end{equation}  

Suppose now that $w$ satisfies \eqref{E: infinity subharmonic} in the viscosity sense in $\Omega$.  Given $V \subset \subset \Omega$ open, $c > 0$, and $x_{0} \notin V$, we claim that
	\begin{equation} \label{E: cone comparison result}
		\sup \left\{ w(x) - c \varphi(x - x_{0}) \, :  \, x \in V \right\} \leq \max \left\{ w(x) - c \varphi(x - x_{0}) \, :  \, x \in \partial V \right\}.
	\end{equation}
We will prove this by making a small perturbation that puts us into the situation described by Proposition~\ref{P: main elliptic} and Proposition~\ref{P: elliptic 2d norm}.
	\vs
Notice that, by Lemma \ref{L: infinity harmonic cones}, the function $C : \mathbb{R}^{d} \setminus \{x_{0}\} \to \mathbb{R}$ given by $C(x) = c \varphi(x - x_{0})$ satisfies $- \langle D^{2} C \cdot \partial \varphi^{*}(DC), \partial \varphi^{*}(DC) \rangle \geq 0$ in $\mathbb{R}^{d} \setminus \{0\}$. 
\vs
 Furthermore, a direct computation shows that, for small $\epsilon$, the perturbation $C^{\epsilon}$ given by $C^{\epsilon}(x) = c \varphi(x - x_{0}) - \epsilon \varphi(x - x_{0})^{2}$ satisfies, for $\epsilon > 0$ small enough,
	\begin{align*}
		- \langle D^{2} C^{\epsilon} \cdot \partial \varphi^{*}(DC^{\epsilon}),  &\partial \varphi^{*}(DC^{\epsilon}) \rangle \\
		&= - (c - 2\epsilon C)  \langle D^{2} C \cdot \partial \varphi^{*}(DC), \partial \varphi^{*}(DC) \rangle + 2\epsilon \varphi^{*}(DC)^{2} \geq \epsilon.\
	\end{align*}
Thus, Proposition~\ref{P: main elliptic} and Proposition~\ref{P: elliptic 2d norm} imply that
	\begin{equation*}
		\sup \left\{ w(x) - C^{\epsilon}(x) \, :  \, x \in V \right\} \leq \sup \left\{ w(y) - C^{\epsilon}(y) \, :  \, y \in \partial V \right\}.
	\end{equation*}
	
Since this is true for any $\epsilon > 0$ small enough, \eqref{E: cone comparison result} follows. \end{proof}

\subsection*{A characterization of distance functions}  The following boundary value problem arises in the $p \to \infty$  limit of certain variational problems in $W_{0}^{1,p}$: 
	\begin{equation} \label{E: distance}
				\min \left\{ \varphi^{*}(Du) - 1, -\langle D^{2}u \cdot \partial \varphi^{*}(Du), \partial \varphi^{*}(Du) \rangle \right\} = 0 \ \ \text{in} \ \ \Omega, \quad u = 0 \ \ \text{on} \ \ \partial \Omega. 
	\end{equation}

It was shown in  \cite{ishibashi_koike} that \eqref{E: distance} has a unique solution when $\varphi$ is the $\ell^{1}-$norm and $\Omega$ is convex.  In what follows, we extend this to the class of norms considered earlier.  
\vs
Recall that, if $\varphi$ is a Finsler norm in $\mathbb{R}^{d}$ and $\Omega \subseteq \mathbb{R}^{d}$ is open, then the function $\text{dist}_{\varphi}(\cdot,\partial \Omega)$ defined by
	\begin{equation*}
		\text{dist}_{\varphi}(x,\partial \Omega) = \inf \left\{ \varphi(x - y) \, :  \, y \in \partial \Omega \right\}
	\end{equation*}
is the unique viscosity solution of the associated Eikonal equation
	\begin{equation} \label{E: eikonal}
				\varphi^{*}(Dd) - 1 = 0 \ \  \text{in} \ \ \Omega \ \ \text{and} \ \ 
				d = 0  \ \  \text{on} \ \ \partial \Omega.
	\end{equation}
Furthermore, the subadditivity of $\varphi$ gives  that, for each open subset $\mathcal{O} \subseteq \Omega$,
	\begin{equation*}
		\text{dist}_{\varphi}(x,\partial \Omega) = \inf \left\{ \text{dist}_{\varphi}(y,\partial \Omega) + \varphi(x - y) \, :  \, y \in \partial \mathcal{O} \right\} \quad \text{if} \, \, x \in \mathcal{O}.
	\end{equation*}

	\begin{theorem} \label{T: distance function}  Let $\varphi$ be a Finsler norm satisfying \eqref{A: running assumptions}.  Given a bounded open set $\Omega \subseteq \mathbb{R}^{d}$, if $w \in USC(\Omega)$ is a viscosity subsolution of \eqref{E: distance} and $v \in LSC(\Omega)$, a supersolution, then 
		\begin{equation*}
			w \leq \text{dist}_{\varphi}(\cdot,\partial \Omega) \leq v \quad \text{in} \, \, \Omega.
		\end{equation*}
	Furthermore, $\text{dist}_{\varphi}(\cdot,\partial \Omega)$ is the unique viscosity solution.  \end{theorem}  
	
		\begin{proof} First, as mentioned above, $d = \text{dist}_{\varphi}(\cdot,\partial \Omega)$ being a solution of \eqref{E: eikonal} is also a subsolution of \eqref{E: distance}.  Further, by definition, $d$ is the infimum of a family of cones, each of which is a $\varphi$-infinity harmonic function in $\Omega$.  Therefore, $- \langle D^{2}d \cdot \partial \varphi^{*}(Dd), \partial \varphi^{*}(Dd) \rangle \geq 0$ in $\Omega$.  This proves $d$ is also a viscosity supersolution, hence it is a solution.
\vs		
		Next, if $v$ is a supersolution, then $\varphi^{*}(Dv) - 1 \geq 0$ in $\Omega$.  Therefore, by the comparison principle associated with \eqref{E: eikonal}, we have $v \geq \text{dist}_{\varphi}(\cdot,\partial \Omega)$.
\vs		
		
	Finally, suppose that $w$ is a subsolution.  We claim that, for each $y \in \partial \Omega$,
		\begin{equation*}
			w(x) \leq \varphi(x - y) \quad \text{if} \, \, x \in \Omega.
		\end{equation*}
To prove this, we will modify the function $x \mapsto C(x)=\varphi(x - y)$ so that it is a strict supersolution of \eqref{E: distance} in $\Omega$.
	\vs
	We claim that there is a constant $A > 0$ such that, for each $\epsilon > 0$, we can find a $C^{\epsilon} \in C(\overline{\Omega})$ that is a strict supersolution of \eqref{E: distance} in $\Omega$ and such that $\|C^{\epsilon} - C\|_{L^{\infty}(\Omega)} \leq A \epsilon$.
\vs	
	To see this, define $C^{\epsilon}$ by 
		\begin{equation*}
			C^{\epsilon}(x) = (1 + 2\epsilon) \varphi(x - y) - \frac{1}{2} \epsilon \|C\|_{L^{\infty}(\Omega)}^{-2} \varphi(x - y)^{2}.
		\end{equation*}
	Since $2 - \|C\|_{L^{\infty}(\Omega)}^{-2} \varphi(x - y)^{2} \geq 1$ in $\Omega$, we have
		\begin{equation*}
			\varphi^{*}(DC^{\epsilon}) - 1 \geq \epsilon \ \ \text{in} \ \ \Omega.
		\end{equation*}
	Further, since $C$ is $\varphi$-infinity harmonic in $\Omega$ and $\langle p, \partial \varphi^{*}(p) \rangle = \varphi^{*}(p)$ for all $p \in \mathbb{R}^{d}$, we can compute
		\begin{align*}
			- \langle D^{2} C^{\epsilon} \cdot \partial \varphi^{*}(DC^{\epsilon}), \partial \varphi^{*}(DC^{\epsilon}) \rangle &= - (1 + 2 \epsilon - \|C\|_{L^{\infty}(\Omega)}^{-2} \epsilon) \langle D^{2} C \cdot \partial \varphi^{*}(DC), \partial \varphi^{*}(DC) \rangle \\
			&\quad + \|C\|_{L^{\infty}(\Omega)}^{-2} \epsilon \varphi^{*}(DC)^{2} \geq \|C\|^{-2}_{L^{\infty}(\Omega)} \epsilon.
		\end{align*}
This proves that $C^{\epsilon}$ is a strict supersolution of \eqref{E: distance} in $\Omega$.  
\vs	
	Applying one of Proposition~\ref{P: main elliptic} or Proposition~\ref{P: elliptic 2d norm} to $w$ and the function $x \mapsto C^{\epsilon}(x) + A \epsilon$, we find
		\begin{equation*}
			\sup \left\{ w(x) - C^{\epsilon}(x) \, :  \, x \in \Omega \right\} \leq A \epsilon,
		\end{equation*}
	and, thus,
		\begin{equation*}
			\sup \left\{ w(x) - C(x) \, :  \, x \in \Omega \right\} \leq 2A \epsilon.
		\end{equation*}
	We conclude upon sending $\epsilon \to 0^{+}$.  \end{proof}
		
\subsection*{A nonlinear eigenvalue problem}  Finally, we revisit the principal eigenvalue problem for the Finsler infinity Laplacian \eqref{E: infinity laplace}.  For a given bounded open set $\Omega \subseteq \mathbb{R}^{d}$, this is the boundary value problem	\begin{equation} \label{E: infinity eigenvalue}
		\left\{ \begin{array}{r l}
			\min \left\{ \varphi^{*}(Du) - \Lambda u, - \langle D^{2}u \cdot \partial \varphi^{*}(Du), \partial \varphi^{*}(Du) \right\} = 0 & \text{in} \ \ \Omega, \\
			u = 0 & \text{on} \ \ \partial \Omega,
		\end{array} \right.
	\end{equation}
which has been studied when $\varphi^{*} \in C^{1}(\mathbb{R}^{d} \setminus \{0\})$  in  \cite{belloni_kawohl_juutinen}.
\vs
By approximating a given Finsler norm $\varphi$ with smoother ones, it is possible to prove that \eqref{E: infinity eigenvalue} always has a viscosity solution minimizing an $L^{\infty}-$variational problem.

	\begin{prop} \label{P: existence of eigenfunction} Given a Finsler norm $\varphi : \mathbb{R}^{d} \to [0,\infty)$ and a bounded domain $\Omega \subseteq \mathbb{R}^{d}$, there exist $\Lambda_{\infty}(\varphi) > 0$ and a Lipschitz continuous function $u : \overline{\Omega} \to [0,\infty)$ solving \eqref{E: infinity eigenvalue} in the viscosity sense and
		\begin{equation} \label{E: eigenvalue}
			\Lambda_{\infty}(\varphi) = \frac{\|\varphi^{*}(Du)\|_{L^{\infty}(\Omega)}}{\|u\|_{L^{\infty}(\Omega)}} = \min \left\{ \frac{\|\varphi^{*}(Dv)\|_{L^{\infty}(\Omega)}}{\|v\|_{L^{\infty}(\Omega)}} \, :  \, v \in W^{1,\infty}_{0}(\Omega) \right\}.
		\end{equation}
	\end{prop}  
	
When $\varphi$ satisfies our main assumptions \eqref{A: running assumptions}, we can prove that any viscosity solution of \eqref{E: infinity eigenvalue} is positive and $\Lambda$ is unique.

	\begin{theorem} \label{T: well defined eigenvalue}  Given a Finsler norm $\varphi$ satisfying \eqref{A: running assumptions} and a bounded open set $\Omega \subseteq \mathbb{R}^{d}$, if $u : \overline{\Omega} \to [0,\infty)$ is a viscosity solution of \eqref{E: infinity eigenvalue} for some $\Lambda > 0$, then 
		\begin{equation*}
			u > 0 \quad \text{in} \, \, \Omega, \quad \Lambda = \Lambda_{\infty}(\varphi).
		\end{equation*}  \end{theorem}  
		
The proof follows as in \cite{belloni_kawohl_juutinen}, the only new necessary ingredient being Proposition~\ref{P: main elliptic} and Proposition~\ref{P: elliptic 2d norm}.  To start with, if $u$ is a viscosity solution of \eqref{E: infinity eigenvalue} and $v = \log(u)$, then $v$ is a viscosity solution of the PDE:
	\begin{equation} \label{E: hopf cole equation}
		\left\{ \begin{array}{r l}
			\min \left\{\varphi^{*}(Dv) - \Lambda, - \langle D^{2}v \cdot \partial \varphi^{*}(Dv), \partial \varphi^{*}(Dv) - \varphi^{*}(Dv)^{2} \rangle \right\} = 0 & \text{in} \ \ \Omega, \\[1.2mm]
			\lim_{\delta \to 0^{+}} \max\left\{v(x) \, :  \, \text{dist}(x,\partial \Omega) \leq \delta \right\} = -\infty.
		\end{array} \right.
	\end{equation}

Here, exactly as in \cite{juutinen_lindqvist_manfredi}, we can prove a comparison principle for \eqref{E: hopf cole equation} provided at least one is finite on the boundary.

\begin{theorem} \label{T: infinity eigenvalue comparison} If $\varphi$ is a Finsler norm satisfying \eqref{A: running assumptions}, $\Omega \subseteq \mathbb{R}^{d}$ is a bounded open set, and $(w,v) \in USC(\Omega) \times LSC(\Omega)$ satisfy
	\begin{align*}
		\min \left\{ \varphi^{*}(Dw) - \Lambda, - \langle D^{2}w \cdot \partial \varphi^{*}(Dw), \partial \varphi^{*}(Dw) \rangle - \varphi^{*}(Dw)^{2} \right\} \leq 0 \ \ \text{in} \ \ \Omega \\
		\min \left\{ \varphi^{*}(Dv) - \Lambda, - \langle D^{2}v \cdot \partial \varphi^{*}(Dv), \partial \varphi^{*}(Dv) \rangle - \varphi^{*}(Dv)^{2} \right\} \geq 0 \ \ \text{in} \ \ \Omega, \\
		\lim_{\delta \to 0^{+}} \sup \left\{ w(x) - v(y) \, :  \, \|x - y\| + \text{dist}(x,\partial \Omega) + \text{dist}(y,\partial \Omega) \leq \delta \right\} \leq 0,
	\end{align*}
then $w \leq v$ in $\Omega$.  \end{theorem} 

\begin{proof}  As shown in \cite{belloni_kawohl_juutinen}, we can find a strict positive supersolution $\tilde{v}_{\epsilon}$ such that $\|\tilde{v}_{\epsilon} - v\|_{L^{\infty}(\Omega)} \leq 2\epsilon$.   Applying the propositions to $\tilde{v}$ and $\tilde{w}_{\epsilon} + 3\epsilon$, we find
	\begin{equation*}
		\sup \left\{ w(x) - \tilde{v}_{\epsilon}(x) \, :  \, x \in \Omega \right\} \leq 3\epsilon.
	\end{equation*}
and, thus,
	\begin{equation*}
		\sup \left\{ w(x) - v(x) \, :  \, x \in \Omega \right\} \leq 5 \epsilon.
	\end{equation*}
We conclude upon sending $\epsilon \to 0^{+}$.  \end{proof}  

Using Theorem \ref{T: infinity eigenvalue comparison}, the proof of Theorem \ref{T: well defined eigenvalue} follows as in \cite{belloni_kawohl_juutinen}.

\appendix

\section{Dimension Two}  \label{Appendix: 2d}

When $d = 2$, the simple geometry allows  to easily cook up nice norms to use in a comparison proof.  This enables us to generalize the comparison principle to operators adapted to arbitrary norms provided we impose slightly stronger assumptions.
\vs

In what follows, we  consider operators $\overline{F}, \underline{F} : \mathbb{R}^{2} \times \mathcal{S}^{2} \to \mathbb{R}$ with the property  that there is a sequence $(e_{n})_{n \in \mathbb{N}} \subseteq S^{1}$ such that 
	\begin{align}
		\lim_{N \to \infty} \sup \left\{ \frac{\overline{F}(ae_{n},X) - \underline{F}(ae_{n},X)}{1 + \|X\|} \, :  \, a > 0, \, \, n \geq N, \, \, X \in \mathcal{S}^{2} \right\} = 0, \label{A: small variation} \\
		\overline{F}(ae_{n},X) = \underline{F}(ae_{n}, X) \quad \text{if} \, \, a > 0, \, \, n \in \mathbb{N}, \, \, X \in \mathcal{S}_{\langle e_{n} \rangle}, \label{A: flat 2d} \\
		\overline{F}(p,X) = \underline{F}(p,X) \ \  \text{if} \ \ (p,X) \in \left(\mathbb{R}^{2} \setminus \bigcup_{n = 1}^{\infty} \mathbb{R}_{+} \{e_{n}\} \right) \times \mathcal{S}^{2}. \label{A: consistency 2d}
	\end{align}
Below we show that these assumptions apply, in particular, to an arbitrary infinity Laplace operator.
\vs	
The following theorem is a slight variation of the one appearing in \cite{level_set}.

\begin{theorem} \label{T: 2D comparison} Suppose that $\overline{F}, \underline{F} : \mathbb{R}^{2} \times \mathcal{S}^{2} \to \mathbb{R}$ satisfy \eqref{A: elliptic_1}, \eqref{A: elliptic_2}, \eqref{A: semicontinuous}, \eqref{A: small variation}, and \eqref{A: consistency 2d}.  If $w \in USC(\mathbb{R}^{2} \times (0,T))$ and $v \in LSC(\mathbb{R}^{2} \times (0,T))$ are bounded and
		\begin{align*}
			w_{t} - \overline{F}(Dw,D^{2}w) \leq 0 \ \ \text{in} \ \ \mathbb{R}^{2} \times (0,T), \\
			v_{t} - \underline{F}(Dv,D^{2}v) \geq 0 \ \ \text{in} \ \ \mathbb{R}^{2} \times (0,T), \\
			\lim_{\delta \to 0^{+}} \sup \left\{ w^{*}(x,0) - v_{*}(y,0) \, :  \, \|x - y\| \leq \delta \right\} \leq 0,
		\end{align*}
	then $w \leq v$ holds in $\mathbb{R}^{2} \times (0,T)$.  Furthermore, for each $u_{0} \in BUC(\mathbb{R}^{2})$, there is a unique bounded viscosity solution $u$ of \eqref{E: parabolic}.\end{theorem}  
	
Since the theorem follows by arguing exactly as in \cite{level_set}, we do not repeat  the proof here.  The idea is that \eqref{A: small variation} implies that, up to a small error, $\overline{F}$ and $\underline{F}$ coincide at all but finitely many points.  For operators with finitely many discontinuities, it is easy to find Finsler norms $\varphi$ such that \eqref{A: consistency} holds. However,  \cite{level_set} shows how to construct a sequence of such norms that exploit \eqref{A: small variation}.
\vs
Arguing similarly, we obtain the following comparison result for elliptic problems.

\begin{prop} \label{P: elliptic 2d norm} Under the assumptions of Theorem \ref{T: 2D comparison}, if $U \subseteq \mathbb{R}^{2}$ is a bounded open set, $f : \Omega \to (0,\infty)$ is a positive continuous function, and $(w,v) \in USC(\Omega) \times LSC(\Omega)$ satisfy
	\begin{gather*}
		- \overline{F}(Dw,D^{2}w) \leq 0 \ \ \text{in} \ \ \Omega, \quad -\underline{F}(Dv,D^{2}v) \geq f \ \ \text{in} \ \ \Omega, \\
		\lim_{\delta \to 0^{+}} \sup \left\{ w(x) - v(y) \, :  \, \|x - y\| \leq \delta, \, \, \text{dist}(x,\partial \Omega) + \text{dist}(y,\partial \Omega) \leq \delta \right\} \leq 0,
	\end{gather*}
then 
	\begin{equation*}
		\sup \left\{ w(x) - v(x) \, :  \, x \in \Omega \right\} \leq 0.
	\end{equation*}
\end{prop}  

\subsection*{Application: Equations involving infinity Laplacian operators}  For our purpose, the main applications of Theorem \ref{T: 2D comparison} and Proposition \ref{P: elliptic 2d norm} are equations involving the infinity Laplacian operator with respect to an arbitrary Finsler norm $\varphi$.  All that we need to justify this assertion is to verify that, if $\varphi$ is a Finsler norm in $\mathbb{R}^{2}$, the pair $(\overline{F},\underline{F})$ satisfies \eqref{A: small variation}, \eqref{A: flat 2d}, and \eqref{A: consistency 2d}.  
\vs
The basic point that we need is the following lemma.

	\begin{lemma} \label{L: BV lemma}  If $\psi : \mathbb{R}^{2} \to [0,\infty)$ is a Finsler norm and $\mathscr{S} = \{q \in S^{1} \, :  \, \# \partial \psi(q) > 1\}$ is the set of directions in which $\psi$ is not differentiable, then $\mathscr{S}$ is countable and, for any $\delta > 0$, there is an $N \in \mathbb{N} \cup \{0\}$ such that		
		\begin{equation*}
			\sup \left\{ \text{diam}(\partial \psi(e)) \, :  \, e \in S^{1} \setminus \{e_{1},\dots,e_{N}\} \right\} \leq \delta.
		\end{equation*}
	\end{lemma}
	
			\begin{proof}  We argue, first, that $\mathscr{S}$ is countable and, second, that, for each $\delta > 0$, we have
			\begin{equation*}
				\# \left\{q \in \mathscr{S} \, :  \, \text{diam}(\partial \psi(q)) \geq \delta \right\} < \infty.
			\end{equation*}
		
		To see that $\mathscr{S}$ is countable, observe that to each $q \in S^{1}$, we can associate a cone $K_{q} = \{x \in \mathbb{R}^{2} \, :  \, x = \alpha p \, \, \text{for some} \, \, p \in \partial \psi(q), \, \, \alpha > 0\}$.  Notice that $K_{q}$ has nonempty interior if and only if $q \in \mathscr{S}$.  In particular, $\{\text{rint}(K_{q}) \, :  \, q \in \mathscr{S}\}$ is a pairwise disjoint family of open subsets of $\mathbb{R}^{2}$.  By the separability of $\mathbb{R}^{2}$, $\mathscr{S}$ must be countable.
\vs		
		Next, observe that, for each $q \in \mathscr{S}$, the subdifferential $\partial \psi(q)$ is a one-dimensional face of $\{\psi^{*} \leq 1\}$ which, being convex,  has finite perimeter.  Accordingly, we can compute
			\begin{equation*}
				\infty > \mathcal{H}^{1}(\{\psi^{*} = 1\}) \geq \sum_{q \in \mathscr{S}} \mathcal{H}^{1}(\partial \psi(q)) = \sum_{q \in \mathscr{S}} \text{diam}(\partial \psi(q)).
			\end{equation*}
		Therefore, for any $\delta > 0$, at most finitely many $q \in \mathscr{S}$ have $\text{diam}(\partial \psi(q)) \geq \delta$.       
		 \end{proof} 
	
The previous lemma implies that, for any Finsler norm $\varphi$ in $\mathbb{R}^{2}$, the infinity caloric equation satisfies the asymptotic consistency condition \eqref{A: small variation}.

\begin{prop} \label{P: 2d case} If $\varphi$ is a Finsler norm in $\mathbb{R}^{2}$, the pair $(\overline{F},\underline{F})$ is given by \eqref{E: infinity laplace}, and, if $(e_{n})_{n \in \mathbb{N}} \subseteq S^{1}$ are the directions in which $\varphi^{*}$ is not differentiable, then \eqref{A: small variation}, \eqref{A: flat 2d}, and \eqref{A: consistency 2d} all hold. \end{prop}  

\begin{proof}  By Lemma \ref{L: BV lemma}, for each $\delta > 0$, we can choose $N \in \mathbb{N}$ such that
	\begin{equation*}
		\sup \left\{ \|q' - q\| \, :  \, q,q' \in \partial \varphi^{*}(e_{n}), \, \, n \geq N \right\} \leq \delta.
	\end{equation*}
Then \eqref{A: small variation} holds, since, by the definition of $(F^{*},F_{*})$,
	\begin{equation*}
		\sup \left\{ \frac{F^{*}(ae_{n},X) - F_{*}(ae_{n},X)}{1 + \|X\|} \, :  \, a > 0, \, \, n \geq N, \, \, X \in \mathcal{S}^{2} \right\} \leq \delta^{2}.
	\end{equation*}
\vs 
For any $n \in \mathbb{N}$, we know that $\mathcal{S}_{\langle e_{n} \rangle} \subseteq \mathcal{S}(e_{n},\varphi)$ by definition.  Therefore, Proposition \ref{P: infinity laplace} implies \eqref{A: flat 2d} directly.
\vs
At the same time, by the choice of $(e_{n})_{n \in \mathbb{N}}$, we know that $\partial \varphi^{*}(p) = \{D\varphi^{*}(p)\}$ if $p \notin \cup_{n = 1}^{\infty} \mathbb{R}_{+} \{e_{n}\}$.  Therefore, \eqref{A: consistency 2d} is an immediate consequence of the definitions. \end{proof}  

Combining Proposition \ref{P: 2d case} and Theorem \ref{T: 2D comparison}, we deduce a comparison result for the infinity caloric equation with respect to an arbitrary norm in $\mathbb{R}^{2}$:

\begin{corollary} \label{C: infinity caloric 2d} Suppose that $\varphi : \mathbb{R}^{2} \to [0,\infty)$ is any Finsler norm.  If $w \in USC(\mathbb{R}^{2} \times (0,T))$ and $v \in LSC(\mathbb{R}^{2} \times (0,T))$ are bounded and satisfy
	\begin{align*}
		w_{t} - \langle D^{2}w \cdot \partial \varphi^{*}(Dw), \partial \varphi^{*}(Dw) \rangle \leq 0 \ \ \text{in} \ \ \mathbb{R}^{2} \times (0,T), \\
		v_{t} - \langle D^{2} v \cdot \partial \varphi^{*}(Dv), \partial \varphi^{*}(Dv) \rangle \geq 0 \ \ \text{in} \ \ \mathbb{R}^{2} \times (0,T), \\
		\lim_{\delta \to 0^{+}} \sup \left\{ w^{*}(x,0) - v_{*}(y,0) \, :  \, \|x - y\| \leq \delta \right\} \leq 0,
	\end{align*}
then $w \leq v$ in $\mathbb{R}^{2} \times (0,T)$.  \end{corollary}

Comparison results for elliptic analogues can also be derived by invoking Proposition \ref{P: 2d case} with Proposition \ref{P: elliptic 2d norm}.  This is fully explained in Section \ref{S: variational} above.  

\section{Auxiliary Computations} \label{Appendix: technical lemmas}

\subsection*{Generalized tangent spaces of the $\ell^{\infty}-$norm}  Consider the case when $\varphi$ is the $\ell^{1}-$norm, that is, $\varphi(q) = \sum_{i = 1}^{d} |q_{i}|$.  In this case, as is well known, $\varphi^{*}(p) = \max_{i} |p_{i}|$.  Let us compute the generalized tangent spaces $\mathcal{T}(\cdot,\varphi)$ and space of matrices $\mathcal{S}(\cdot,\varphi)$ for this example.
\vs
In general, if $\psi$ is any Finsler norm, then $\mathcal{T}(0,\psi) = \{0\}$.  This follows immediately from the fact that $\partial \psi^{*}(0)$ contains a neighborhood of the origin.  From this, we deduce that $\mathcal{S}(0,\psi) = \{0\}$ independently of $\psi$.
\vs
Hence we only need to consider $p \in \mathbb{R}^{d} \setminus \{0\}$.  It will be convenient to introduce some notation.
\vs
First, we let $\{\bar{e}_{1},\dots,\bar{e}_{d}\}$ denote the standard orthonormal basis in $\mathbb{R}^{d}$.  Let $E$ denote the extreme points of $\{\varphi^{*} \leq 1\}$, that is, the set 
	\begin{equation*}
		E = \{\rho \bar{e}_{i} \, \mid \, \rho \in \{-1,1\}, \, \, i \in \{1,2,\dots,d\} \}.
	\end{equation*} 
and define the set-valued maps $J : \mathbb{R}^{d} \to \mathcal{P}(E)$ by 
	\begin{equation*}
		J(p) = \text{argmax}_{e \in E} \langle p, e \rangle, \quad L(p) = \text{argmin}_{e \in E} \langle p,e \rangle.
	\end{equation*} 
\vs
Given $p \in \mathbb{R}^{d} \setminus \{0\}$, an elementary computation shows that the subdifferential $\partial \varphi^{*}(p)$ is determined by
	\begin{equation*}
		\partial \varphi^{*}(p) = \text{conv}(J(p)).
	\end{equation*}
Thus, the generalized tangent space $\mathcal{T}(p,\varphi)$ is
	\begin{equation*}
		\mathcal{T}(p,\varphi) = \partial \varphi^{*}(p)^{\perp} = J(p)^{\perp} = \text{span}(\{e_{k} \, \mid \, |p_{k}| < \varphi^{*}(p)\}).
	\end{equation*}
From this and some elementary linear algebra, we find that the linear subspaces $\langle p \rangle \oplus \mathcal{T}(p,\varphi)$ can be written in the form
	\begin{equation*}
		\langle p \rangle \oplus \mathcal{T}(p,\varphi) = \left \langle \sum_{e \in J(p)} e \right \rangle \oplus \text{span}\{e_{k} \, \mid \, |p_{k}| < \varphi^{*}(p)\}.
	\end{equation*}
Notice that the dimension of these subspaces equals $d - \#J(p) + 1$.  Finally, by definition, $\mathcal{S}(p,\varphi)$ equals the set of symmetric matrices on this subspace.

\subsection*{$F_{1}$ is not a Finsler infinity Laplacian} \label{S: not finsler}  We show that the operator $F_{1}$ of \eqref{E: median operator} cannot be rewritten as a Finsler infinity Laplacian.  We argue by contradiction.
\vs
Suppose that $F_{1}(p,X) = \langle X \partial \varphi^{*}(p), \partial \varphi^{*}(p) \rangle$ for some Finsler norm $\varphi$.  Using the definition of $F_{1}$, this implies that, for each $i \in \{1,2,\dots,d\}$, $\varphi^{*}$ is piecewise linear in the set $\{p \in \mathbb{R}^{d} \, :  \, |p_{i}| < |p_{j}| \, \, \text{if} \, \, j \neq i\}$ and 
	\begin{equation*}
		X_{ii} = F_{1}(p,X) = \langle X D\varphi^{*}(p), D\varphi^{*}(p) \rangle \ \  \text{if} \ \ |p_{i}| < \min\{|p_{1}|,\dots,|p_{i-1}|,|p_{i+1}|,\dots,|p_{d}|\}.  
	\end{equation*}
Note that this holds independently of the choice of $X \in \mathcal{S}^{d}$.  Thus, 
	\begin{equation} \label{E: complete crap}
		D\varphi^{*}(p) \in \{\bar{e}_{i},-\bar{e}_{i}\} \ \  \text{if} \ \ |p_{i}| < \min\{|p_{1}|,\dots,|p_{i-1}|,|p_{i+1}|,\dots,|p_{d}|\}.
	\end{equation}
\vs
Fix an $i \in \{1,2,\dots,d\}$ and $p \in \mathbb{R}^{d} \setminus \{0\}$ with $|p_{i}| <  \min\{|p_{1}|,\dots,|p_{i-1}|,|p_{i+1}|,\dots,|p_{d}|\}$.  It follows from  \eqref{E: subdifferential} and \eqref{E: complete crap} that we can fix a $\rho \in \{-1,1\}$ such that
	\begin{equation*}
		\rho p_{i} = \langle D\varphi^{*}(p), p \rangle = \varphi^{*}(p).
	\end{equation*}
Assuming without loss of generality that $p_{1} > 0$, and replacing $p$ by $p + \alpha \bar{e}_{1}$ for some $\alpha > 0$ give
	\begin{equation*}
		\rho p_{i} = \langle D \varphi^{*}(p + \alpha \bar{e}_{1}), p + \alpha \bar{e}_{1} \rangle = \varphi^{*}(p + \alpha \bar{e}_{1}).
	\end{equation*}
At the same time, since $\varphi^{*}$ is a Finsler norm, 
	\begin{equation*}
		\varphi^{*}(p + \alpha \bar{e}_{1}) \geq \inf \left\{ \frac{\varphi^{*}(p')}{\|p'\|} \, :  \, p' \in \mathbb{R}^{d} \setminus \{0\} \right\} \|p + \alpha \bar{e}_{1}\|.
	\end{equation*}
The desired contradiction now follows, since
	\begin{equation*}
		\rho p_{i} = \lim_{\alpha \to \infty} \varphi^{*}(p + \alpha \bar{e}_{1}) \geq \inf \left\{ \frac{\varphi^{*}(p')}{\|p'\|} \, :  \, p' \in \mathbb{R}^{d} \setminus \{0\} \right\} \cdot \lim_{\alpha \to \infty} \|p + \alpha \bar{e}_{1}\| = \infty.
	\end{equation*}
\vs
\vs

\subsection*{Acknowledgements}  The second author was partially supported by the National Science Foundation grant
DMS-1900599, the Office for Naval Research grant N0001417- 12095, and the Air
Force Office for Scientific Research grant FA9550-18-1-0494. The first author was 
supported by the second author's National Science Foundation grant
DMS-1900599.

  \bibliographystyle{plain}
\bibliography{bibliography}

\begin{thebibliography}{10}

\bibitem{armstrong_crandall_julin_smart}
Scott~N Armstrong, Michael~G Crandall, Vesa Julin, and Charles~K Smart.
\newblock Convexity criteria and uniqueness of absolutely minimizing functions.
\newblock {\em Archive for rational mechanics and analysis}, 200(2):405--443,
  2011.

\bibitem{armstrong_smart_easy_proof}
Scott~N. Armstrong and Charles~K. Smart.
\newblock An easy proof of {J}ensen's theorem on the uniqueness of infinity
  harmonic functions.
\newblock {\em Calc. Var. Partial Differential Equations}, 37(3-4):381--384,
  2010.

\bibitem{aronsson_crandall_juutinen}
Gunnar Aronsson, Michael~G. Crandall, and Petri Juutinen.
\newblock A tour of the theory of absolutely minimizing functions.
\newblock {\em Bull. Amer. Math. Soc. (N.S.)}, 41(4):439--505, 2004.

\bibitem{barles_georgelin}
Guy Barles and Christine Georgelin.
\newblock A simple proof of convergence for an approximation scheme for
  computing motions by mean curvature.
\newblock {\em SIAM J. Numer. Anal.}, 32(2):484--500, 1995.

\bibitem{belloni_kawohl_juutinen}
M.~Belloni, B.~Kawohl, and P.~Juutinen.
\newblock The {$p$}-{L}aplace eigenvalue problem as {$p\to\infty$} in a
  {F}insler metric.
\newblock {\em J. Eur. Math. Soc. (JEMS)}, 8(1):123--138, 2006.

\bibitem{belloni_kawohl}
Marino Belloni and Bernd Kawohl.
\newblock The pseudo-{$p$}-{L}aplace eigenvalue problem and viscosity solutions
  as {$p\to\infty$}.
\newblock {\em ESAIM Control Optim. Calc. Var.}, 10(1):28--52, 2004.

\bibitem{chatterjee_souganidis}
Sourav Chatterjee and Panagiotis~E Souganidis.
\newblock Convergence of deterministic growth models.
\newblock {\em arXiv preprint arXiv:2108.00538}, 2021.

\bibitem{CGG}
Yun~Gang Chen, Yoshikazu Giga, and Shun'ichi Goto.
\newblock Uniqueness and existence of viscosity solutions of generalized mean
  curvature flow equations.
\newblock {\em J. Differential Geom.}, 33(3):749--786, 1991.

\bibitem{crandall_gunnarsson_wang}
Michael~G Crandall, Gunnar Gunnarsson, and Peiyong Wang.
\newblock Uniqueness of $\infty$-harmonic functions and the eikonal equation.
\newblock {\em Communications in Partial Differential Equations},
  32(10):1587--1615, 2007.

\bibitem{user}
Michael~G. Crandall, Hitoshi Ishii, and Pierre-Louis Lions.
\newblock User's guide to viscosity solutions of second order partial
  differential equations.
\newblock {\em Bull. Amer. Math. Soc. (N.S.)}, 27(1):1--67, 1992.

\bibitem{di-castro_perez-llanos_urbano}
Agnese Di~Castro, Mayte P\'{e}rez-Llanos, and Jos\'{e}~Miguel Urbano.
\newblock Limits of anisotropic and degenerate elliptic problems.
\newblock {\em Commun. Pure Appl. Anal.}, 11(3):1217--1229, 2012.

\bibitem{durier}
Roland Durier.
\newblock On {P}areto optima, the {F}ermat-{W}eber problem, and polyhedral
  gauges.
\newblock {\em Math. Programming}, 47(1, (Ser. A)):65--79, 1990.

\bibitem{EvansSpruck}
L.~C. Evans and J.~Spruck.
\newblock Motion of level sets by mean curvature. {I}.
\newblock {\em J. Differential Geom.}, 33(3):635--681, 1991.

\bibitem{gurtin_soner_souganidis}
M.~E. Gurtin, H.~M. Soner, and P.~E. Souganidis.
\newblock Anisotropic motion of an interface relaxed by the formation of
  infinitesimal wrinkles.
\newblock {\em J. Differential Equations}, 119(1):54--108, 1995.

\bibitem{ishibashi_koike}
Toshihiro Ishibashi and Shigeaki Koike.
\newblock On fully nonlinear {PDE}s derived from variational problems of
  {$L^p$} norms.
\newblock {\em SIAM J. Math. Anal.}, 33(3):545--569, 2001.

\bibitem{ishii}
Hitoshi Ishii.
\newblock Degenerate parabolic {PDE}s with discontinuities and generalized
  evolutions of surfaces.
\newblock {\em Adv. Differential Equations}, 1(1):51--72, 1996.

\bibitem{IshiiSoug}
Hitoshi Ishii and Panagiotis Souganidis.
\newblock Generalized motion of noncompact hypersurfaces with velocity having
  arbitrary growth on the curvature tensor.
\newblock {\em Tohoku Math. J. (2)}, 47(2):227--250, 1995.

\bibitem{juutinen_kawohl}
Petri Juutinen and Bernd Kawohl.
\newblock On the evolution governed by the infinity {L}aplacian.
\newblock {\em Math. Ann.}, 335(4):819--851, 2006.

\bibitem{juutinen_lindqvist_manfredi}
Petri Juutinen, Peter Lindqvist, and Juan~J. Manfredi.
\newblock The {$\infty$}-eigenvalue problem.
\newblock {\em Arch. Ration. Mech. Anal.}, 148(2):89--105, 1999.

\bibitem{level_set}
Peter~S Morfe.
\newblock On the homogenization of second order level set {PDE} in periodic
  media.
\newblock {\em arXiv preprint arXiv:2011.15062}, 2020.

\bibitem{ohnuma_sato}
Masaki Ohnuma and Moto-Hiko Sato.
\newblock Singular degenerate parabolic equations with applications to
  geometric evolutions.
\newblock {\em Differential Integral Equations}, 6(6):1265--1280, 1993.

\bibitem{perez-llanos_rossi}
Mayte Perez-Llanos and Julio~D. Rossi.
\newblock An anisotropic infinity {L}aplacian obtained as the limit of the
  anisotropic {$(p,q)$}-{L}aplacian.
\newblock {\em Commun. Contemp. Math.}, 13(6):1057--1076, 2011.

\bibitem{rockafellar}
Ralph~Tyrell Rockafellar.
\newblock {\em Convex analysis}.
\newblock Princeton University Press, 2015.

\bibitem{rossi_saez}
Julio~D. Rossi and Mariel Saez.
\newblock Optimal regularity for the pseudo infinity {L}aplacian.
\newblock {\em ESAIM Control Optim. Calc. Var.}, 13(2):294--304, 2007.

\bibitem{schneider}
Rolf Schneider.
\newblock {\em Convex bodies: the Brunn--Minkowski theory}.
\newblock Number 151. Cambridge university press, 2014.

\bibitem{steinhaus}
Hugo Steinhaus.
\newblock {\em Mathematical snapshots}.
\newblock Oxford University Press, 1960.

\bibitem{mathworld_article}
Eric~W. Weisstein.
\newblock Rhombic dodecahedron.
\newblock \url{https://mathworld.wolfram.com/RhombicDodecahedron.html}.
\newblock Accessed: 2021-08-30.

\end{thebibliography}

\end{document}